\documentclass[aop,preprint]{imsart}

\RequirePackage[OT1]{fontenc}
\RequirePackage{amsthm,amsmath}
\RequirePackage[numbers]{natbib}
\RequirePackage[colorlinks,citecolor=blue,urlcolor=blue]{hyperref}

\arxiv{1703.06922}

\startlocaldefs
\numberwithin{equation}{section}
\theoremstyle{plain}
\newtheorem{thm}{Theorem}[section]
\endlocaldefs

\usepackage{bbm}
\usepackage{latexsym, amssymb, amsmath, amscd, amsthm, amsxtra}
\usepackage{mathtools}
\usepackage{enumerate}
\usepackage[all]{xy}
\usepackage{mathrsfs}
\usepackage{fancyhdr}
\usepackage{listings}
\usepackage{hyperref}
\usepackage[normalem]{ulem}

\def\P{\mathbb{P}}
\def\Pr{\mathrm{\mathbf P}}
\def\Ex{\mathrm{\mathbf E}}
\def\E{\mathbb{E}}

\def\Z{\mathbb{Z}}
\def\R{\mathbb{R}}
\def\N{\mathbb{N}^*}

\def\11{\mathbbm{1}}

\def\pe{\mathsf{p}}
\def\tran{\mathsf{T}}

\def\K{\mathcal{K}}
\def\D{\mathcal{D}}
\def\calU{\mathcal{U}}

\def\calC{\mathcal{C}}
\def\calV{\mathcal{V}}
\def\calM{\mathcal{M}}
\def\calW{\mathcal{W}}

\def\smt{{\mathsf{v}_*}}

\def\ft{{T}}
\def\ob{\mathcal{O}}

\def\locV{\mathsf{V}}
\def\locD{\D_*}
\def\loc1D{\D_*}
\def\tD{ \D}
\def\alo{\alpha_1}
\def\alt{\alpha_1}
\def\alh{\alpha_2}

\def\swp{\gamma}
\newcommand\dif{\mathop{}\!\mathrm{d}}

\newtheorem{proposition}[thm]{Proposition}

\newtheorem{lemma}[thm]{Lemma}
\newtheorem{cor}[thm]{Corollary}
\newtheorem{defn}[thm]{Definition}

\theoremstyle{definition}
\newtheorem{remark}[thm]{Remark}

\numberwithin{equation}{section}

\begin{document}
\definecolor{qing}{RGB}{0, 153, 153}

\begin{frontmatter}
\title{Poly-logarithmic localization for random walks among random obstacles}
\runtitle{Random walks with random obstacles}

\begin{aug}
\author{\fnms{Jian} \snm{Ding}\thanksref{t1}\ead[label=e1]{dingjian@wharton.upenn.edu}} \and
\author{\fnms{Changji} \snm{Xu}\thanksref{t1}\ead[label=e2]{changjixu@galton.uchicago.edu}}

\thankstext{t1}{Partially supported by NSF grant DMS-1455049, DMS-1757479 and an Alfred Sloan fellowship.}
\runauthor{J. Ding and C. Xu}

\affiliation{University of Pennsylvania and University of Chicago}

\address{Department of Statistics,\\
 The Wharton School, \\
The University of Pennsylvania,\\
 Philadelphia, PA 19104\\
\printead{e1} }

\address{Department of Statistics, \\
The University of Chicago,\\
 Chicago, IL 60637\\
\printead{e2}}
\end{aug}

\begin{abstract}
Place an obstacle with probability $1-\pe$ independently at each vertex of $\mathbb Z^d$, and run a simple random walk until hitting one of the obstacles. For $d\geq 2$ and $\pe$ strictly above the critical threshold for site percolation, we condition on the environment where the origin is contained in an infinite connected component free of obstacles, and we show that the following \emph{path localization} holds for environments with probability tending to 1 as $n\to \infty$: conditioned on survival up to time $n$ we have that
ever since $o(n)$ steps the simple random walk is localized in a region of volume poly-logarithmic in $n$ with probability tending to 1. The previous best result of this type went back to Sznitman (1996) on Brownian motion among Poisson obstacles, where a localization (only for the end point) in a region of volume $t^{o(1)}$ was derived conditioned on the survival of Brownian motion up to time $t$.
\end{abstract}

\begin{keyword}[class=MSC]
\kwd{60K37, 60H25, 60G70.}
\end{keyword}

\begin{keyword}
\kwd{random walk among random obstacles}
\kwd{localization.}
\end{keyword}

\end{frontmatter}

\section{Introduction}

For $d\geq 2$, we consider a random environment where  each vertex of $\mathbb Z^d$ is occupied by an obstacle independently with probability  $1- \pe \in (0,1)$. Given this random environment, we then consider a discrete time simple random walk $(S_t)_{t \in \mathbb{N}}$ started at the origin and killed at the first time $\tau$ when it hits an obstacle. In this paper, we study the quenched behavior of the random walk conditioned on survival for a large time, and we prove the following localization result. For convenience of notation, throughout the paper we use $\P$ (and $\E$) for the probability measure with respect to the random environment, and use $\Pr$ (and $\Ex$) for the probability measure with respect to the random walk.

\begin{thm}\label{thm-main}
  For any fixed $d\geq 2$ and $\pe > p_c(\mathbb Z^d)$ (the critical threshold for site percolation), we condition on the event that the origin is in an infinite cluster (i.e., infinite connected component) free of obstacles. Then 
  there exists a constant $c = c(d,\pe)$ and a collection of $\mathbb P$-measurable subsets $D_n \subset \Z^d $ of cardinality at most $(\log n)^{c}$ and of distance at least $n (\log n)^{-100d^2}$ from the origin, such that  the following holds.

  There exists a random time $T \in [0,cn(\log n)^{-2/d}]$ such that
   \begin{equation}
\Pr\Big(T \leq c|S_T|, S_{[T,n]} \subset D_n\mid \tau > n\Big) \to 1 \mbox{ in } \P\mbox{-probability}\,.\\
   \end{equation} 
   Here $S_{[T,n]}$ is used to denote for $\{S_t:  t\in [T,n]\}$.
\end{thm}   

We say a vertex is open if it is free of obstacle, and thus each vertex is open with probability $\pe$. In this paper,  we have chosen $\pe>p_c(\mathbb Z^d)$ to make sure that with positive probability the open cluster containing the origin is infinite.    A variation of the model is to place obstacles on edges rather than on vertices, and one can easily verify that our result as well as the its proof extend to that case. In addition, in this paper, the obstacles are chosen to be hard, i.e., killing a random walk with probability 1. One could alternatively consider soft obstacles. That is, every time the random walk hits an obstacle it has a certain fixed probability (which is strictly less than 1) to be killed. Further, one could also consider the continuous analogue, i.e., Brownian motion with Poissonian obstacles as in \cite{Sznitman98}. While we believe our methods useful in these settings, we leave these for future study.

\subsection{Background and related results}

Random walks among random obstacles has been studied extensively in literature. In the annealed case, the logarithmic asymptotics for survival probabilities are closely related to the large deviation estimates for the range of the random walk \cite{DR75, DR79, Sznitman90, Sznitman93}. Indeed, it boils down to the following optimization problem:
\begin{equation}\label{eq-annealed}
\max_{k}\P(|S_{[0, n]}| = k) \pe^k\,,
\end{equation} 
where the key to determining the optimizer is an estimate on $\P(|S_{[0, n]}| \leq k)$ when $k$ is substantially smaller than  typical  $|S_{[0, n]}|$.
 The localization problem has also been studied in the annealed case, where in \cite{Bolthausen94,Sznitman91} it was proved that in dimension two the range of the random walk/Wiener sausage is asymptotically a (Euclidean) ball and in \cite{Povel99} it was shown that in dimension three and higher the range of the Brownian motion is contained in a ball; in both cases the asymptotics of the radius for the balls were determined. The behavior that the range of the random walk/Wiener sausage is asymptotically a  ball, is fundamentally determined by the celebrated Faber--Krahn inequality (which states that among sets in $\mathbb R^d$ with given volume balls are the only sets which minimize the first Dirichlet eigenvalue). Indeed, in \cite{Bolthausen94} a key ingredient was a quantitative version of this type of inequality in $\mathbb Z^2$ which was proved in the same paper; in \cite{Sznitman97} another quantitative version of Faber--Krahn was proved independently;  in \cite{Povel99} a quantitative version of isoperimetric inequality (related to Faber--Krahn inequality) from \cite{Hall92} was a key ingredient in the proof. Provided that the range is asymptotically a ball, in order to determine the radius of the ball one just need to solve \eqref{eq-annealed} with $\P(|S_{[0, n]}| = k)$ replaced by $\P(S_{[0, n]}\mbox{ is contained in a ball of volume } k)$ --- this is then a relatively straightforward computation.

The localization problem in the quenched case was far more challenging: in the Brownian setting it was studied first in \cite{Sznitman93b, Sznitman97} and its logarithmic asymptotics for the survival probability was first derived in \cite{Sznitman93b} (see also the celebrated monograph \cite{Sznitman98}). In \cite{Fukushima09} a simple argument for the quenched asymptotics of the survival probability was given using the Lifshitz
tail effect.  In the random walk setting, the logarithmic asymptotics of survival probability was computed in \cite{Antal95} which built upon methods developed in the Brownian setting.  Stating the result of \cite{Antal95} in our setting, we have the following: conditioned on the origin being in the infinite open cluster, one has that with $\P$-probability tending to 1 as $n\to \infty$
\begin{equation}\label{eq:000}
\Pr(\tau > n) = \exp\{-c_*n(\log n)^{-2/d}(1 + o(1))\}\,.
\end{equation}
Here $c_* = \mu_{B} (\frac{\omega_d \log p}{d})^{2/d}$, $B$ is a unit ball in $\R^d$, $\omega_d$ is the volume of  $B$ and $\mu_{B}$ is the first eigenvalue of the Dirichlet-Laplacian of $B$ which is formally defined as 
$$\mu_B = \frac{1}{2d}\min_{u\in W_0^{1, 2}(B)} \{\int_B |\nabla u|^2 \dif x: \|u\|_{L^2(B) =1}\}$$
($W_0^{1, 2}(B)$ is the closure of $C_0^\infty(B)$ in the norm $\|u\|_{W_0^{1, 2}(B)} = (\int_B |\nabla u|^2 dx)^{1/2}$).
We remark that Faber--Krahn inequality also plays a fundamental role in \eqref{eq:000} --- one strategy to obtain the lower bound in \eqref{eq:000} is for the random walk to travel (in minimal possible number of steps) to the largest open ball within chemical distance (i.e., graph distance in open clusters) $n^{1-o(1)}$ from the origin  and stays within that ball afterwards. 

 In terms of localization, an analogous result to Theorem~\ref{thm-main} was obtained in \cite{Sznitman96} (see also \cite{Sznitman98}) with an upper bound of $t^{o(1)}$ on the volume of localizing region. Such region is sometimes called (an union of) \emph{islands}, where an island is a connected subset in $\mathbb Z^d$.
Theorem~\ref{thm-main} substantially improves the previous best result of this type \cite{Sznitman96, Sznitman98} in  the following two aspects. 
\begin{itemize}
 \item In \cite{Sznitman96, Sznitman98} the author studied Brownian motion $(B_t)$ among Poissonian obstacles and showed that $B_t$ is localized in $t^{o(1)}$ many islands each of which has volume $t^{o(1)}$, conditioned on survival up to time $t$. In comparison, Theorem~\ref{thm-main} yields a localization in poly-logarithmic in $n$ many islands each of which has volume poly-logarithmic in $n$, conditioned on the random walk surviving up to time $n$. 
\item Path localization was proved in \cite{Sznitman96, Sznitman98} in one-dimensional case, that is to say, the trapping time (i.e., the time it takes to reach an island in which the Brownian motion/random walk stays  afterwards) is sub-linear in $t$. But in dimension two or higher no path localization was derived. We show that for $d\geq 2$ the trapping time is linear in the Euclidean distance between the island (where the localization occurs) and the origin --- this is a strong path localization result which in particular implies that the trapping time is sublinear in $n$. 
 \end{itemize}
 In addition, we believe that the strategy we described below \eqref{eq:000} to achieve the (asymptotic) lower bound in \eqref{eq:000} is \emph{essentially} the optimal strategy for the random walk. Since the largest open ball under consideration has order $\log n$, we thus believe that the size of $D_n$ should be of order $\log n$. With this belief, our upper bound is expected to be sharp in the nature of poly-logarithmic but not sharp in terms of the power. 

Thus far, we have been discussing random walks with Bernoulli obstacles, i.e., at each vertex we kill the random walk with probability either 0 or a certain fixed number. More generally, one may place i.i.d.\ random potentials $\{W_v: v\in \mathbb Z^d\}$ (where $W_v$ follows a general distribution) and one assigns a random walk path probability proportional to $\exp(\sum_{i=0}^n W_{S_i})$. The case of Bernoulli obstacles is a prominent example in this family. Previously, there has been a huge amount of work devoted to the study of various localization phenomenon when the potential distribution exhibits some tail behavior ranging from heavy tail to doubly-exponential tail. 
See \cite{Wolfgang16} for an almost up-to-date review on this subject, also known as the parabolic Anderson model and the random Schr\"odinger operators. See also \cite{ADS17} for a review on random walk among mobile/immobile random traps. 

For a very partial review, in the works of \cite{GMS83, AS07, AS08, HMS08, KLMS09, LM12, ST14}, much progress has been made for heavy-tailed potentials. In particular they proved localization in a single lattice point \cite{KLMS09,FM14,LM12,ST14} for potentials with tails heavier than doubly-exponential. We note that by localization in a single lattice point we meant for a single large $t$, as considered in the present article; one could alternatively consider the behavior for all large $t$ simultaneously as in \cite{KLMS09}, in which case they showed that the random walk is localized in two lattice points, almost surely as $t\to \infty$. In a few recent papers \cite{BK16, BKS16} (which improved upon \cite{GKM07}), the case of doubly-exponential potential was tackled where detailed behavior on leading eigenvalues and eigenfunctions, mass concentration as well as aging were established. In particular, it was proved that in the doubly-exponential case the mass was localized in a bounded neighborhood of a site that achieves an optimal compromise between the local Dirichlet eigenvalue of the Anderson Hamiltonian and the distance to the origin.
\subsection{Future directions} Provided with the present article, there are a number of natural future directions (e.g., proving an analogue of our result in the case for Brownian motion and/or for soft obstacles, as mentioned right after Theorem~\ref{thm-main}). Here we list a few problems of substantial interests.
\begin{itemize}
\item Show that there exists a (conjecturally) unique island for localization.
\item Determine the asymptotic volume and shape for this island.
\item Show that $\max_{x \in \Z^d} \Pr(S_n = x \mid \tau > n)= O(1/\log n)$.
\item Determine the order of the distance between this island and the origin (we note that Theorem~\ref{thm-main} gives that such distance is between $\Omega(n (\log n)^{-100d^2})$ and $O(n (\log n)^{-2/d})$). 
\item Describe the geometry of the range for the random walk.
\end{itemize}
Based on our current understanding, we believe that completely solving the aforementioned questions will require a number of new ideas and we expect both results and techniques of the present paper to play an important role.

\subsection{A word on proof strategy}\label{sec:proof-strategy}
We will consider small regions whose volume is poly-logarithmic in $n$, and consider their principal eigenvalues (formally, the principal eigenvalue for a region is the largest eigenvalue for the transition kernel of the random walk killed upon hitting an obstacle or exiting the region).
The starting point of our proof is the crucial intuition that localization more or less amounts to the phenomenon that the order statistics for principal eigenvalues in small regions which are within distance $n$ from the origin have non-small spacings near the edge (i.e., near the extremum). Non-small spacings for principal eigenvalues near the edge plays an important role in controlling the number of (which turns out to be at most poly-logarithmic in $n$) small regions where the random walk will be localized in: Since the spacings are non-small near the edge, this \emph{roughly speaking} implies that any small region which is not one of the best poly-logarithmic in $n$ regions, is strictly suboptimal compared to the best small region. That is to say, the random walk would prefer to stay in the best small region instead of the union of all the suboptimal regions. In other words, the best poly-logarithmic in $n$ regions are the only possible regions for which the random walk would spend a substantial amount of time. This implies the poly-logarithmic localization as desired. Next, we describe our proof strategy in more detail.

Since principal eigenvalues in small regions are more or less i.i.d., such spacings near the edge are determined by the tail behavior of principal eigenvalues: the heavier the tail is, the larger the spacing is near the edge. To implement this intuition, we consider the survival probability after a poly-logarithmic number of steps starting from each vertex in the box of size $n$ --- such survival probabilities are closely related to principal eigenvalues in a region of poly-logarithmic diameter (see Lemma~\ref{eigv}). Here we have to choose the number of steps $k_n$ at least logarithmic in $n$, otherwise we will have too many starting points with survival probability 1. What is important to us, is the fact that by choosing $k_n$ poly-logarithmic in $n$, we already get a tail on such survival probabilities which is heavy enough for our purpose. 

In light of the above discussions, a key task is to prove that the survival probability, viewed as a random variable measurable with respect to the random environment, has non-light tails. This is incorporated in Section~\ref{sec:tail}. Note that there are many balls of radius $10^{-3}(\log_{1/\pe} n)^{1/d}$ which are free of obstacles and thus have atypically high survival probabilities for random walks started inside. Thus, in light of our interest in spacings only near the edge of the order statistics, it suffices to control the right tail of the survival probability that is far away from its typical value. For vertices started from which the survival probabilities in $k_n$ steps are high, we can then {\it a priori} prove that the random walk spends at least a positive fraction of steps in a set of cardinality $O(\log n)$ near this vertex (see Proposition~\ref{cgoodtime}). This implies that there exists at least one vertex with large local times conditioned on survival in $k_n$ steps. Therefore, by removing the closest obstacle near this vertex we will be able to add a significant fraction of paths and thus significantly improve the survival probability. Finally, by controlling the cardinality of the preimage of this operation of removing an obstacle, we obtain the desired tail behavior on survival probability, as shown in the proof of Proposition~\ref{tail}.

With Proposition~\ref{tail} at hand, we can then show in Lemma~\ref{goodsiteprob} that there are poly-logarithmic many local regions that are candidates for localization, and any other regions have significantly lower survival probabilities compared to the best candidate regions. Combined with well-known tools from percolation theory, a positive fraction of the candidate regions are connected to the origin by open paths of lengths which are linear in their Euclidean distances from the origin. This is the content of Section~\ref{sec:candidate-regions}.

Using ingredients from Section~\ref{sec:candidate-regions}, we prove in Lemma~\ref{firststep} that conditioned on survival the random walk with probability close to 1 visits one of the candidate regions, for the reason that moving to the best reachable candidate region quickly and staying there afterwards yields a much larger survival probability than never visiting any of the candidate regions. Next, we prove in Proposition~\ref{endpointloc} that once the random walk reaches a candidate region it is not efficient to move far away without entering another candidate region. Up to this point, we have derived the poly-logarithmic localization as desired.

Finally, it remains to show that the amount of time for the random walk to reach the region in which it is localized afterwards is at most linear in the Euclidean distance of this region from the origin. To this end, we employ the notion of loop erasure for the random walk, and show that the size of the loop erasure is at most linear and that the total size of (erased) loops is also at most linear. This is the content of Section~\ref{pathlocalization}.

\subsection{Comparison to earlier works}

In previous works on localization for parabolic Anderson models, one way to derive localization was from rather precise information on leading eigenvalues, see, e.g., \cite{BKS16}. However, it is usually non-trivial to compute the asymptotics for leading principal eigenvalues. It occurs to us that with this approach in order to derive a poly-logarithmic localization, one would have to compute the principal eigenvalues in substantially higher precision than that in \cite{Sznitman98, Fukushima09}, which seems to be challenging.

Another way is to derive localization from (lower bound on) fluctuations of principal eigenvalues, and such method was used in \cite{Sznitman97,Sznitman98} to control the distance of the localization region from the origin. In particular, in \cite{Sznitman97} the author studied the following variational problem (for large $t$)
\begin{equation}\label{eq-variational-problem}
F_t (\ell)= \ell + t \lambda_\ell
\end{equation}
where $\lambda_\ell$ is the principal Dirichlet eigenvalue of $\tfrac{1}{2} \Delta$ for $(-\ell, \ell)^d$. In \cite[Theorem 3.2]{Sznitman97} a lower bound on the fluctuation of principal eigenvalues \emph{averaging over many scales} was derived, from which information on minimizers of $\ell$ for \eqref{eq-variational-problem} was derived in \cite[Theorem 5.1]{Sznitman97}.  In addition, a similar variational problem has been studied in \cite{Sznitman98} (see \cite[Page 281, Equation (3.1)]{Sznitman98}), via which a localization phenomenon was derived.  However, along the arguments of \cite{Sznitman97, Sznitman98}, the total volume of the localization region one could derive was of order $t^{o(1)}$ as implied in \cite{Sznitman98}. One of a few reasons is that the number of minimizers for \eqref{eq-variational-problem} could potentially be large.

Our method shares the same underlying philosophy which emphasizes the crucial role of fluctuations of principal eigenvalues, as  discussed in Section~\ref{sec:proof-strategy}. However, our proof is mostly self-contained and in particular does not borrow from \cite{Sznitman98, Antal95, Sznitman97}.  Somewhat surprisingly, our proof does not rely on the estimate  \eqref{eq:000} either. In comparison with \cite{Sznitman97, Sznitman98} the following features of our method are crucial for deriving a poly-logarithmic localization.
\begin{itemize}
\item We directly work with principal eigenvalues of small regions. In comparison, working with principal eigenvalues in the big box $(-\ell, \ell)^d$ would need an additional step to show the localization of the principal eigenfunctions.
\item We derive a tail estimate on principal eigenvalues of small regions, which allows us to control spacings near the edge of their order statistics in \emph{every} scale (instead of averaging over many scales as in \cite{Sznitman97}).
\item We directly compare the probabilities for different set of paths instead of working with variational problems such as \eqref{eq-variational-problem}, which avoids introducing extra error factors in the analysis.
\end{itemize}

\subsection{Notation convention}

For notation convenience, we denote by $\ob$ the collection of all obstacles (sometimes referred to as closed vertices) and $\mathcal C(v)$ the open cluster containing $v$. 

For $A\subset \mathbb Z^d$,  write $\partial A = \{x\in A^c: y\sim x \mbox{ for some } y\in A\}$, where $x\sim y$ means that $x$ is a neighbor of $y$ and $\partial_{i} A = \{x\in A: y\sim x \mbox{ for some } y\in A^c\}$. We denote by $\xi_A = \inf\{t \geq 0: S_t \not\in A\}$ the first time for the random walk to exit from $A$, and by $\tau_A = \inf\{t \geq 0: S_t\in A\}$ the hitting time to $A$. As having appeared earlier, we write $\tau = \tau_{\ob}$ for the survival time of the random walk. 

For $m \in \N = \{1, 2, 3, \ldots\}$, we denote by  $S_{[0, m]} = \{S_0, \ldots, S_m\}$  the range of the first $m$ steps of the random walk.  A path is a sequence of vertices $\omega = [\omega_0,\omega_1,...,\omega_{|\omega|}]$ where 
$|\omega|$ is its length and $\omega_i$, $\omega_{i+1}$ are adjacent for $0\leq i \leq |\omega|-1$. We say a path is open if all of its vertices are open. For $u,v \in \Z^d$, we say $u\leftrightarrow v$ if there exists an open path that connects $u$ and $v$. We define the chemical distance by 
\begin{equation}\label{eq-def-chemical}
D(u,v) = \inf\{|\omega|: \omega_0 = u, \omega_{|\omega|}=v, \omega \mbox{ is open}\}\,.
\end{equation}
We denote the $\ell^2$-distance by $|u-v| = \big(\sum_{i=1}^d (u_i - v_i)^2 \big)^{1/2}$.   We denote discrete $\ell^2$-ball by $B_r(v) = \{x \in \Z^d: |x-v| \leq r \}$.

We write $A_n \lesssim B_n$ if there exits a constant $C>0$ depending only on $(d,\pe)$ such that $A_n \leq C B_n$ for all $n$, and $A_n \gtrsim B_n$ if $B_n \lesssim A_n$. If $A_n \gtrsim B_n$ and $A_n \lesssim B_n$, we write $A_n \asymp B_n $. 

Here is a list of the rest of the symbols used in this paper, followed by the place of their definition.
\begin{center}
\begin{tabular}{ llllll } 
 \hline
 $X_v$& \eqref{eq:def-xv}&
 $k_n$& \eqref{eq-def-k-n}&
 $c$-good & \textsc{Def.} \ref{def-c-good}\\ 
 $\epsilon$-fair & \textsc{Def.}\ref{def-epsilon-fair} & 
 $\mathcal K(\cdot,\cdot)$&\eqref{eq-def-K}&
 $p_{\alpha}$'s& \eqref{eq:palpha-def}\\
 $\calU_\alpha$'s& \eqref{eq:U-def}&$\mathcal C_R(v)$& \textsc{Def.} \ref{def:CRlambda}& $\lambda_v,R$& \textsc{Def.} \ref{def:CRlambda}\\
 $\locD$, $\tD$'s & \eqref{eq-def-D}& $\locV$ & \eqref{eq:def-locv}& $D_n$& \eqref{eq:def-dn}\\
$\smt$& \textsc{Def.} \ref{def:vt}& $T(\cdot)$& \eqref{eq-def-T}&
$\eta$& \eqref{eq:defLE}  \\
$\calM(t)$& \textsc{Def.} \ref{def-Mt}& $A_t(\omega)$& \textsc{Def.} \ref{def-Mt}& &\\
 \hline
\end{tabular}
\end{center}

\medskip

\noindent {\bf Acknowledgment.} We warmly thank Ryoki Fukushima, Steve Lalley and Rongfeng Sun for valuable discussions. We warmly thank Rongfeng and Ryoki for a careful reading on an earlier manuscript with numerous valuable comments. We also thank Alain-Sol Sznitman for pointing out \cite{Sznitman97} which we missed in an earlier version. We warmly thank two anonymous referees for numerous helpful comments on the exposition of our manuscript.  Finally, much of the work was carried out when J.D. was at University of Chicago.

\section{Tail behavior of survival probabilities} \label{sec:tail}

The main goal of this section is to prove right tail bounds on the survival probability, as incorporated in Proposition~\ref{tail} below (see also the discussions below Proposition~\ref{tail} for its proof strategy). To this end, for each vertex $v \in  \mathbb Z^d$, we let
\begin{equation}
\label{eq:def-xv}
	X_v = \Pr^v(\tau > k_n)
\end{equation}
be the probability that the random walk started at $v$ survives for at least $k_n$ steps, where $k_n$ is set as (we denote by $\lfloor x \rfloor$ the greatest integer less than or equal to $x$ for $x\in \mathbb R$)
\begin{equation} \label{eq-def-k-n}
k_n = \begin{cases}
    \lfloor (\log n)^3 (\log \log n)^2 \rfloor & \mbox{if } d = 2;\\
    \lfloor(\log n)^{4 - 2/d }\rfloor & \mbox{if } d \geq 3.
  \end{cases} 
  \end{equation}
  We remark that there is no fundamental reason for our choice of $k_n$: it has to be poly-logarithmic in $n$ so that it is ``small'', and it has to be at least substantially larger than $(\log n)^{2/d}$ so that  $\max_{v \in B_n(0)}X_v = o(1)$. We made our particular choice of $k_n$ for convenience of analysis.
Note that $X_v$ is a $\P$-measurable random variable. As mentioned in the introduction, it suffices to consider the right tail of $X_v$ far away from its typical value. For reasons that will become clear soon, it is convenient to set the threshold as   
$$ \beta_\chi = \chi^{k_n/(\log n)^{2/d}},$$
where $\chi$ is a positive constant to be selected.
\begin{lemma}\label{lem-lambdac}
There exists $\chi = \chi(d,\pe) > 0$ such that 
\begin{equation}
  \label{eq:lambdac}
  \P(X_v \geq \beta_\chi) \gtrsim n^{-d + 1}\,.
\end{equation}
\end{lemma}
\begin{proof}
Since $|B_r(v)| \asymp r^d$, there exists $c_{d,\pe}$ depending only on $(d, \pe)$ such that
\begin{equation}\label{eq-open-ball}
\P(B_{c_{d,\pe} (\log n)^{1/d}}(v) \subset \ob^c) \gtrsim n^{-d + 1}.
\end{equation}
When all vertices in $B_{c_{d,\pe} (\log n)^{1/d}}(v)$ are open, the random walk with initial point $v$ will survive in $k_n$ steps if it stays in $B_{c_{d,\pe} (\log n)^{1/d}}(v)$. Next, we estimate the probability for the random walk to stay in a ball. This is a fairly simple and standard argument, which we give only for completeness. It is clear that there exists $c= c(d)>0$ such that 
$$\min_{x\in B_r} \Pr(S_t \in B_{2r} \mbox{ for } 0\leq t<r^2, S_{r^2} \in B_r) \geq c$$
for all $r\geq 1$. Now, set $r =  \lfloor 2^{-1}c_{d,\pe} (\log n)^{1/d} \rfloor$. By having the random walk to stay within $B_{2r}$ and to end in $B_r$ for every block of $r^2$ steps, we obtain 
\begin{equation} \label{eq-rw-ball}
\Pr^v(S_t \in B_{2r}(v), t= 0,1,...,k_n)
      \geq c^{ k_n/(r^2) +1}\,.
      \end{equation}
Now, we can choose $\chi = \chi(d, \pe)>0$ small enough so that $c^{ k_n/(r^2) +1} \geq \beta_\chi$. Combining \eqref{eq-open-ball} and \eqref{eq-rw-ball}, we complete the proof of the lemma.
\end{proof}
\begin{remark}
A sharp version of \eqref{eq-rw-ball} with the exact large deviation rate was derived in \cite{DR79}, but we do not need such sharp estimate here.
\end{remark}
Lemma~\ref{lem-lambdac} justifies our choice of considering the right tail of $X_v$ only above the threshold $\beta_\chi$ for some small $\chi>0$, since there is at least one site $v \in B_n(0)$ with $X_v \geq \beta_\chi$ and thus the extremal level set is above $\beta_\chi$. In what follows, we always choose $\chi>0$  such that \eqref{eq:lambdac} holds (and it will become clear that eventually we will choose a $\chi>0$ depending only on $(d, \pe)$).
\begin{proposition}
\label{tail}
 For all $\chi>0$ and $\beta \geq \beta_\chi$, we have 
  \begin{equation}
  \label{eq:tail}
    \P(X_v \geq \beta) \leq c_1 k_n^d \P(X_v \geq c_2 \beta\log n) + n^{-(2d+1)}\,.
  \end{equation}
  where $c_1, c_2$ are positive constants only depends on $(d, \pe, \chi)$.
\end{proposition}

The proof of Proposition~\ref{tail} consists of two main ingredients:
\begin{enumerate}[(a)]
\item The random walk spends a positive fraction of steps in a subset of size $O(\log n)$ conditioned on survival (in the case when the survival probability is at least $\beta_\chi$). Thus, there exists at least one vertex $x$ which is visited for many times on average conditioned on survival.
\item If we change the environment by removing the closest obstacle around $x$ we will increase the survival probability substantially, and this will lead to the desired tail estimate \eqref{eq:tail}. 
\end{enumerate}
We now describe how we prove (a), i.e., to control the support of the local times for the random walk.
\begin{itemize}
\item We first show in Proposition~\ref{cgoodtime} that conditioned on survival (in the case when the survival probability is at least $\beta_\chi$), the random walk spends at least $k_n/2$ steps on $c$-good vertices (c.f. Definition~\ref{def-c-good}).
\item Next we show in Lemma~\ref{cgoodperco} that each $c$-good vertex has to be contained in a ``connected'' component of $\epsilon$-fair boxes (c.f. Definition~\ref{def-epsilon-fair}) of volume at least  $\Omega(\log n)$.
\item Since $\epsilon$-fair only occurs with small probability by Lemma~\ref{epsilonfair}, we use a percolation type of argument in Lemma~\ref{cgoodprob} to show that $c$-good occurs very rarely, and then in Lemma~\ref{lem-c-good-bound} that the number of $c$-good vertices is  $O(\log n)$. 
\end{itemize}

The ``environment changing'' argument as in (b) is carried out in the \emph{Proof of Proposition~\ref{tail}}, which itself is divided into three steps. One can see the discussions at the beginning of \emph{Proof of Proposition~\ref{tail}} for an outline of its implementation. 

\subsection{Support of local times} This subsection is devoted to proving (a), following the three steps outlined above.

\begin{defn}\label{def-c-good}
A site $v$ in $\Z^d$ is called $c$-good if \begin{equation}
\label{eq:cgood-def}
	\Pr^v(\tau > \lfloor (\log n )^{2/d} \rfloor ) \geq c\,.
\end{equation} 	
\end{defn}
 We first show that the random walk tends to spend many steps on $c$-good vertices. 
\begin{proposition}
\label{cgoodtime}
For any $\chi>0$, there exists $c = c(\chi)>0$ such that for all environments: 
$$\Pr(\tau > k_n, |\{t\leq k_n: S_t \mbox{ is a $c$-good site}\}| \leq k_n/2) \leq \beta_\chi/2\,.$$
\end{proposition}
\begin{proof}
Let $\zeta_0= -1$ and for $m\geq 1$ recursively define 
$$\zeta_m = \inf\{t \geq \zeta_{m-1}+ (\log n)^{2/d}; S_t \mbox{ is not $c$-good site}\}\,.$$
Write $j_n = \lfloor k_n/(2 (\log n)^{2/d}) \rfloor$. By strong Markov property, we get that
\begin{align*}
   \Pr &(\zeta_{j_n}  \leq k_n < \tau ) 
   \leq \Pr ( S_{[\zeta_{m},\zeta_{m}+ \lfloor (\log n)^{2/d} \rfloor]} \mbox{ is open }\forall 1 \leq m \leq j_n-1) \leq c^{j_n-1}\,.
\end{align*}
Note that on the event $E = \{\tau > k_n, |\{t\leq k_n: S_t \mbox{ is a $c$-good site}\}| \leq k_n/2\}$, we have $\zeta_{j_n}  \leq  k_n< \tau$. Thus, we have $\Pr(E) \leq c^{j_n-1}$. Choosing an appropriate 
$c = c(\chi)$ completes the proof of the proposition.
\end{proof}

Next we control the size of $c$-good vertices. For this purpose, we consider disjoint boxes \begin{equation}
\label{eq:def-kbox}
  K_{r}(x) \coloneqq \{y\in \mathbb Z^2: \|x-y\|_\infty \leq r\}
\end{equation} for $x \in (v+ (2r+1)\Z^d)$ and $r>0$ to be selected. 
\begin{defn}\label{def-epsilon-fair}
	A box $K_r(x)$ is called $\epsilon$-fair if there exist $u\in K_r(x)$ such that 
\begin{equation}
\label{eq:fair}
	\Pr^u(\tau \geq r^2 \mbox{ or } \tau > \xi_{K_{2r}(x)}) \geq \epsilon\,.
\end{equation}
\end{defn}
In what follows, we carry out the last two steps in the outline of proving (a): we show in Lemma~\ref{epsilonfair} that the $\epsilon$-fair boxes are rare provided $r = r(\epsilon)$ large enough, and in Lemma~\ref{cgoodperco} we show that a $c$-good point has to be in a cluster consisting of $\Omega(\log n)$ many $\epsilon$-fair boxes. Combining these two lemmas, we can then bound the probability for a vertex to be $c$-good as in Lemma~\ref{cgoodprob}, which leads to Lemma~\ref{lem-c-good-bound} on the $O(\log n)$ bound for the number of $c$-good vertices in  a  box of radius $k_n$.
\begin{lemma}\label{epsilonfair}
For any $\epsilon>0$, there exists $r = r(\epsilon,d, \pe)$ such that 
\begin{equation}\label{eq-epsilon-fair}
\P(K_r(x) \mbox{ is  $\epsilon$-fair}) \leq \epsilon\,.
\end{equation}
\end{lemma}

\begin{proof}
Let $y$ be an arbitrary vertex in $K_r(x)$. 
By the independence of the environment and random walk, we have    
\begin{align}\label{eq-Y-1}
  \E\left[\Pr^y(|S_{[0,r^2]}| >r^{1/2},\tau>r^2)\right] = \P \otimes\Pr^y(|S_{[0,r^2]}| >r^{1/2},\tau>r^2)\leq \pe^{r^{1/2}-1}\,.
\end{align}
and
\begin{align}\label{eq-Y-2}
  \E\left[\Pr^y(\tau>\xi_{K_{2r}(x)})\right] = \P \otimes\Pr^y(\tau>\xi_{K_{2r}(x)})\leq \pe^{r}\,.
\end{align}

In addition, note that in every $r$ steps the random walk has a positive probability to visit at least $r^{1/2}$ distinct sites. Thus,
there exists a constant $c >0$ such that
$$\Pr^y(|S_{[0,r^2]}|\leq r^{1/2}) \leq  e^{-cr} \,.$$
Combined with \eqref{eq-Y-1} and \eqref{eq-Y-2} it implies that 
$$\sum_{y\in K_r(x)}\P\Big(\Pr^y(\tau \geq r^2 \mbox{ or } \tau > \xi_{K_{2r}(x)}) \geq \epsilon \Big) \leq (2r+1)^d \epsilon^{-1} (\pe^{r^{1/2}-1} +\pe^{r}+ e^{-cr} )\,.$$
Choosing $r = r(\epsilon, d, \pe)$ large enough completes the proof of the lemma.
\end{proof}
We will always choose 
\begin{equation}\label{eq-epsilon}
\epsilon =  \min(c/2, (2d)^{-3^{d+1}}) \mbox{ and } r = r(d,\epsilon, \pe)
\end{equation}
 such that \eqref{eq-epsilon-fair} holds.  We fix $v\in \mathbb Z^d$ and define the adjacency relation for $\epsilon$-fair boxes $\{K_r(x),~ x\in (v+ (2r+1)\Z^d)$ to be the following:
\begin{equation}\label{eq-epsilon-fair-neighboring}
	K_r(x) \sim K_r(y) \iff \exists x' \in K_r(x), y'\in K_r(y)  \ s.t. \ x' \sim y'.
\end{equation}
We next show that in order for a vertex $v$ to be $c$-good, it requires $v$ to be in a cluster consisting of $\Omega(\log n)$ many $\epsilon$-fair boxes --- here a cluster is a connected component where each ``vertex''  corresponds to an $\epsilon$-fair box and the neighboring relation is given by \eqref{eq-epsilon-fair-neighboring}. Thus, $c$-good is a rare event. 
To this end, let  $L_v $ be the subset of $(v+ (2r+1)\Z^d)$ such that $\{K_r(x),~ x\in L_v\}$ is the cluster of $\epsilon$-fair boxes in $B_{h (\log n)^{1/d}}(v)$ which contains $v$. 
\begin{lemma}
\label{cgoodperco}
For any $c>0$ and $\epsilon$ satisfying \eqref{eq-epsilon}, 
there exist $l = l(d,c, \epsilon)$ and $h = h(d,c, \epsilon)$ such that $v$ is not a $c$-good vertex if $|L_v| \leq l\log n$.
\end{lemma}
\begin{proof}
For any $d\geq 2$, there exists a constant $\theta = \theta(d)$ such that \[
  \sup_{x \in \Z^d}\Pr^v(S_t = x) \leq \theta t^{-d/2}~~~\mbox{for all }t \geq 1.
\]
Let $m = \lfloor (\log n)^{2/d}\rfloor -r^2$. For all $0<\Delta<1$, there exists a constant $h = h(\Delta)$ such that
\begin{equation}\label{eq-in-the-ball} 
\Pr^v(S_{[0,m]} \subset B_{2^{-1}h (\log n)^{1/d}}(v))\geq 1 - \Delta\,.
\end{equation}
 In addition, if $|L_v| \leq l\log n$ we have 
  \begin{align*}
    \Ex^{v}\left[\sum_{i=1}^{m} \11_{\{S_i \in  \cup_{x \in L_v} K_r(x)\}}\right] & = \sum_{i=1}^{m} \Pr^v(S_i \in \cup_{x \in L_v} K_r(x) ) \\
    &\leq \lfloor\Delta m \rfloor + \sum_{i=\lfloor\Delta m \rfloor + 1}^{m}|\cup_{x \in L_v} K_r(x)|\theta i^{-d/2} \\
 & \leq \Delta m  + \theta(m - \lfloor\Delta m \rfloor) (2r+1)^d|L_v|\lfloor\Delta m +1 \rfloor^{-d/2}\\
 &\leq \Delta m + \theta\Delta^{-d/2}(2r+1)^d 2^{d/2}ml\,,
  \end{align*} 
  where the last inequality holds when $\log n \geq (2r)^d$. Setting 
  $$0<\Delta < (c - \epsilon)/(3-3\epsilon) \mbox{ and }l = 2^{-d/2}\Delta^{1 + d/2} \theta^{-1}(2r+1)^{-d}\,,$$ we get that
  $$ \Ex^{v}\left[\sum_{i=1}^{m} \11_{\{S_i \in \cup_{x \in L_v} K_r(x) \}}\right]  \leq 2\Delta m\,.$$ 
This implies that $\Pr^v(\xi' > m) \leq 2\Delta$, where $\xi' = \inf\{t \geq 0: S_t \not \in \cup_{x \in L_v} K_r(x)\}$. Combined with \eqref{eq-in-the-ball}, it yields that  with probability at least $1-3\Delta$ the event $ \{ S_{[0,m]} \subset B_{2^{-1}h (\log n)^{1/d}}(v), \xi' \leq m \}$ occurs. Further, on this event, we have that $S_{\xi'}$ is not in an $\epsilon$-fair box, and thus $\Pr^{S_{\xi'}}(\tau > r^2) \leq \epsilon$. 
Therefore,
\begin{equation*}
\Pr^v(\tau>\lfloor (\log n)^{2/d}\rfloor) \leq 1 - (1- 3\Delta)(1-\epsilon) <c\,. \qedhere
\end{equation*}
\end{proof}

\begin{lemma}
\label{cgoodprob}
There exists $\delta = \delta(c,d, \pe)>0$ such that 
$$ \P(|L_v| > l\log n) \leq n^{-\delta}. $$
\end{lemma}

\begin{proof}
For all $x \in L_v$, the number of points $y \in L_v$ such that $K_{2r}(x) \cap K_{2r}(y) \not = \varnothing$ is at most $3^d$.
Therefore, there exists a subset $I$ of $L_v$, such that $|I| \geq |L_v|/3^d$ and that $K_{2r}(x) \cap K_{2r}(y) = \varnothing$ for different $x,y \in L_v$. Hence events $\{K_r(x)$ is $\epsilon$-fair$\}$ for $x\in I$ are independent of each other. In addition, the number of connected components of $|L_v|$ boxes is no more than $(2d)^{2|L_v|}$ --- this is a fairly standard combinatorial computation and one could see, e.g., \cite{WS90} for a reference.
Therefore,
\begin{align*}
  \P(|L_v|> l\log n)  \leq \sum_{j \geq l \log n}(2d)^{2 j} \cdot \epsilon^{j/3^d} \leq 2 n^{l \log(4d^2 \epsilon^{1/3^d})}\,.
\end{align*}
Combined with Lemma~\ref{cgoodperco} and \eqref{eq-epsilon}, it completes the proof of the lemma.
\end{proof}

\begin{lemma}\label{lem-c-good-bound}
For all $v \in \mathbb Z^d$ and $\kappa > 0$,
$$\P(K_{k_n}(v) \mbox{ contains more than } \kappa\log n  \mbox{ $c$-good points}) \leq n^{-\kappa \delta/(4h)^d}\,. $$
\end{lemma}
\begin{proof}
Write $q = 2\lceil h(\log n)^{1/d}  + 2r\rceil$.
Let $K_i = K_{k_n}(v) \cap (i + q \Z^d)$ for $i \in  \{1, \ldots, q \}^d$. 
For any fixed $i$, the events $\{|L_v| \geq l\log n\}$ for $v \in K_i$ are independent. Thus, for large $n$
\begin{align*}
  &\P(K_{k_n}(v) \mbox{ contains more than } \kappa \log n  \mbox{ $c$-good points})  \\
  \leq & \sum_{i\in\{1, \ldots, q\}^d} \P(|\{v \in K_i: |L_v| \geq l\log n \}| \geq \kappa/(3h)^d)\\
 \leq&  q^d\P(\mathrm{Bin}(k_n^d,n^{-\delta})  \geq \kappa/(3h)^d)\,,
\end{align*}
where the last inequality follows from Lemma~\ref{cgoodprob} and $\mathrm{Bin}(k_n^d,n^{-\delta})$ is a binomial random variable with probability $n^{-\delta}$ and $k_n^d$ trials. At this point, the desired bound follows from a standard large deviation estimate for Binomial random variables. 
\end{proof}

\begin{lemma}
\label{probgv}
For $v\in \mathbb Z^d$, let $G_v = G_v(\alpha, \kappa)$ be the event that 

\noindent (1) For every $u\in K_{k_n}(v)$, there exists a closed site within distance $\alpha (\log n)^{1/d}$.

\noindent (2) The number of $c$-good points in $K_{k_n}(v)$ is at most $\kappa \log n$.

Then there exist $\kappa, \alpha>0$ depending only on $(c, d, \pe)$  such that
 $$\P(G_v) \geq 1 - n^{-(2d+1)}\,.$$
\end{lemma}
\begin{proof}
Since $|B_{ \alpha (\log n)^{1/d}}(x)| \geq (\alpha/d)^d \log n$, we have that
\begin{align*}
\P&(\mbox{There exists } x \in K_{k_n}(v) \mbox{ such that } B_{ \alpha (\log n)^{1/d}}(x) \cap  \ob= \varnothing) \\
&\leq (2k_n + 1)^d \pe^{ (\alpha/d)^d \log n}\,.
\end{align*}
This addresses the first requirement for the event $G_v$. The second requirement for $G_v$ is addressed in Lemma~\ref{lem-c-good-bound}. Altogether, we conclude that $\P(G_v) \geq 1 - n^{-(2d+1)}$ with appropriate choices of $\alpha$ and $\kappa$, as desired.
\end{proof}

\subsection{Environment changing argument} We now prove Proposition~\ref{tail}.

\begin{proof}[Proof of Proposition \ref{tail}] We choose $c = c(\chi)$ as in Proposition~\ref{cgoodtime}.
The proof of Proposition \ref{tail} consists of three steps as follows:
\begin{enumerate}
\item For each $c$-good point $x$, removing its closest obstacle would enlarge the survival probability by a factor at least $\ell_x (\log n)^{2/d-2} (\log \log n)^{-1}$ (where the $\log \log n$ terms only appears when $d = 2$). Here we denote by $\ell_x = \Ex\left[\sum_{t=0}^{k_n} \11_{ \{S_t = x\} } \mid \tau > k_n\right]$ the expected number of visits to $x$ conditioned on survival.
\item Combining Step 1 with Proposition~\ref{cgoodtime}, we show that there exists at least one $c$-good point $x$ such that removing the closest obstacle near $x$ enlarge the survival probability by a factor of order $\log n$.
\item The operation of removing the closest obstacle has preimage with multiplicity bounded by $O(k_n^d)$, which leads to the term of $c_1k_n^d$ in \eqref{eq:tail}.
\end{enumerate}

We now carry out the proof steps outlined as above. 

\noindent {\bf Step 1.}  For each $c$-good site $x \in K_{k_n}(v)$, let $x'$ be one of the closed sites nearest to $x$ (with respect to the Euclidean distance) and let $x^*$ be one of the neighbors of $x'$ such that $|x - x^*| <|x-x'|$ (so $x^*$ is open). Let $b = |x - x^*|$, $ \mathring{B}_b(x) = B_b(x) \setminus \{x\}$. For $u, v\in \mathbb Z^d$,  $A\subset \mathbb Z^d$ and $r\geq 1$, we define 
\begin{align}
\begin{split}
\K_{A,r}(u, v)  &= \{\omega = [\omega_0,\ldots, \omega_r]: \omega_0 = u,\omega_r =v,\omega_i \in A \mbox{ for } 1\leq i\leq r-1 \}\,,\\
\K_A(u, v) &= \cup_{r=1}^\infty \K_{A,r}(u, v),\quad \K_{A,r}(u) = \cup_{v \in \Z^d}\K_{A,r}(u, v) \,. \label{eq-def-K}
\end{split}
\end{align}
The key in Step 1 is to construct a collection of paths which does not hit any obstacle except $x'$, such that the collection is large in comparison with the number of paths which does not hit any obstacle.
 To this end, we let $W_x$ be the collection of paths of form $\omega^1 \oplus \pi^1 \oplus [x^*,x',x^*]\oplus \pi^2 \oplus \omega^2$ (here $\oplus$ denotes for the natural concatenation for paths), where $\omega^1, \pi^1, \pi^2, \omega^2$ are ranging over all choices satisfying
 \begin{itemize}
 \item $\omega^1 \in \K_{\ob^c}(v,x), \omega^2 \in \cup_{y\in \ob^c}\K_{\ob^c}(x, y), |\omega^1|+|\omega^2|=k_n$;
 \item $\pi^1 \in \K_{\mathring{B}_b(x)} (x,x^*) ,\pi^2 \in \K_{\mathring{B}_b(x)} (x^*,x)$.
 \end{itemize}

In order to complete Step 1, we need to verify the following two ingredients (which we check below). 
\begin{enumerate}[(a)]
\item We prove that if $\gamma \in W_x$, then the above decomposition into four concatenated parts is unique, and there exists no $\tilde \gamma \in W_x$ which is a continuation of $\gamma$ (meaning, that can be written in the form $\gamma \oplus \pi$ for a non trivial $\pi$).
\item We use this observation to obtain a lower bound on the probability of observing a path in $W_x$ in the first steps of the random walk.
\end{enumerate}

\noindent \underline {\bf Step 1 (a)}.
As $x'$ is visited only once, the separation between $\pi^1$ and $\omega^1$ and that between $\pi^2$ and $\omega^2$ must correspond respectively to the last visit of $x$ before visiting $x$ and the first visit of $x$ after visiting $x'$. This yields uniqueness of the decomposition. The condition $|\omega^1| + |\omega^2|= k_n$ implies that the continuation of a path in $W_x$ cannot belong to $W_x$.

\medskip

\noindent \underline{\bf Step 1 (b).} We will abuse the notation by writing 
$$\Pr^v(W) = \Pr^v([S_0,S_1,...,S_{|\omega|}] = \omega \mbox{ for some } \omega \in W)$$ where $W$ is a collection of paths. Then
\begin{align*}
  \Pr^v(W_x) = \quad \mathclap{\sum_{\omega \in \K_{\ob^c,k_n}(v)}}\quad \  (2d)^{-k_n-2} \Pr^v(\K_{\mathring{B}_b(x)} (x,x^*) )\cdot \Pr^v(\K_{\mathring{B}_b(x)} (x^*,x))\cdot \sum_{i\geq 0}^{k_n} \11_{x}(\omega_i)\,,
\end{align*}
\begin{figure}
	\includegraphics[width=3in]{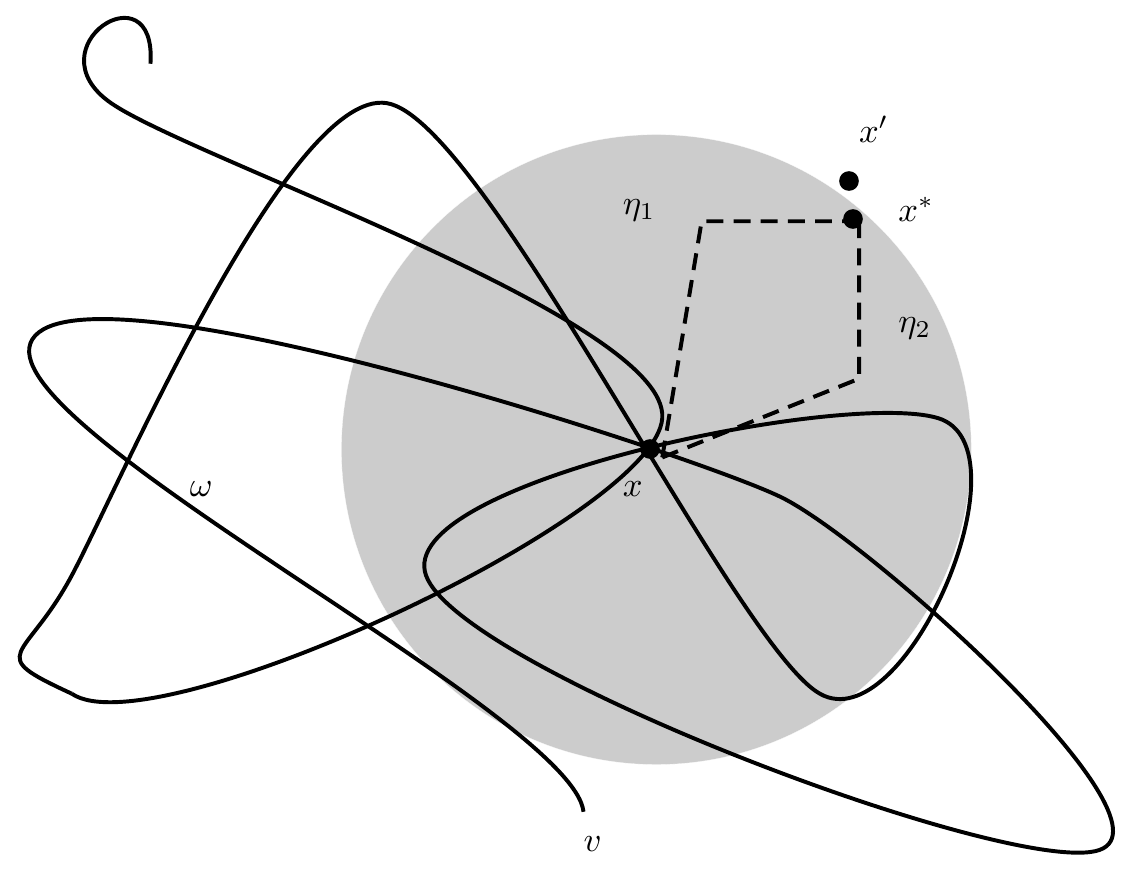}
	\begin{caption}
{The gray ball is $B_b(x)$ which is open. The black curve is the random walk path $\omega$, which visits $x$ three times in this picture. }
\label{fig:Wx}
\end{caption}
\end{figure}where we have the factor $\sum_{i\geq 0}^{k_n} \11_{x}(\omega_i)$ because we can insert an $\pi^1 \oplus [x^*,x',x^*]\oplus \pi^2$ at each visit to $x$ along the random walk path (see Figure \ref{fig:Wx}).
In light of Lemma \ref{probgv}, we choose $\alpha, \kappa$ so that $\P(G_v) \geq 1 - n^{-2d-1}$. And we suppose $G_v$ occurs, hence $b \leq \alpha (\log n)^{1/d}$. By \cite[Lemma 6.3.7]{Lawler10}, we get that 
\begin{align*}
\Pr^v(\K_{\mathring{B}_b(x)} (x,x^*))=\Pr^v(\K_{\mathring{B}_b(x)} (x^*,x))& \geq \begin{cases}
    C(\log n)^{-1/2}(\log \log n)^{-1} & \mbox{if } d = 2;\\
    C(\log n)^{1/d - 1} & \mbox{if } d \geq 3\,.
  \end{cases}
  \end{align*}
where $C>0$ only depends on $\pe$ and $d$. In fact, \cite[Lemma 6.3.7]{Lawler10} gives an estimate on the harmonic measure when the random walk exits a discrete $\ell^2$-ball. Combined with the observation that such harmonic measure is unchanged conditioned on the random walk not returning to the starting point, this yields the preceding inequality. Thus, we have
$$  \Pr(\K_{\mathring{B}_b(x)} (x,x^*)) \geq 2^{-1}Ck_n^{-1/2}\log n\,.$$
Therefore, (recall that $\ell_x = \Ex [\sum_{t=0}^{k_n} \11_{ \{S_t = x\} } \mid \tau > k_n]$) we obtain that
\begin{equation}\label{eq-W-ell}
\Pr(W_x) \geq (2d)^{-3} C^2 k_n^{-1}(\log n)^2 \ell_xX_v\,.
\end{equation}
The preceding inequality can be immediately translated into a bound on the survival probability after removing the obstacle at $x'$.

\noindent {\bf Step 2.}  Recall Proposition \ref{cgoodtime} and recall that $X_v = \Pr^v(\tau > k_n)$.
We see that on the event $\{X_v \geq \beta_\chi\}$ 
$$\sum_{x \in K_{k_n}(v):  x \mbox{ is $c$-good}} \ell_x \geq k_n/4\,.$$
At the same time, on the event $G_v \cap \{X_v \geq \beta_\chi\}$ there are no more than $\kappa\log n$ $c$-good points in $K_{k_n}(v)$. Altogether, it follows that there exists a $c$-good point $x$ in $K_{k_n}(v)$ such that $\ell_x\geq k_n/(4 \kappa\log n)$. Combined with \eqref{eq-W-ell}, it yields that there exists $x' \in K_{k_n}(v)$ such that for $c_2 = c_2(\pe, d)>0$ 
\begin{equation}\label{eq-tau-prime}
\Pr^v(\tau' >k_n) \geq  c_2 X_v\log n\,, \mbox{ where } \tau' = \tau_{\ob\setminus\{x'\}} =  \inf\{t: S_t\in \ob\setminus \{x'\}\}\,.
\end{equation}

\noindent {\bf Step 3.} Now, for $\beta \geq \beta_\chi$, let $E_v$ be the event that $\{$there exists a closed site $x' \in K_{k_n}(v)$ such that 
$\Pr^v(\tau' > k_n) \geq c_2 \beta \log n \}$. We have shown that $G_v \cap \{X_v \geq \beta\} \subset E_v$, as in \eqref{eq-tau-prime}. Furthermore, note that  $x' \in K_{k_n}(v)$ provided $W_x \not= \varnothing$. 
Thus, all environments where $E_v$ occurs can be obtained by closing one open site in $K_{k_n}(v)$ in one of the environments where $X_v \geq c_2 \beta \log n$.  Write $\ob^n = \ob \cap B_{k_n}(v)$ --- we restrict the consideration of $\ob$ into a finite set $B_{k_n}(v)$ so that each realization of $\ob^n$ has positive probability. Let $A, B$ be collections of subsets of $B_{k_n}(v)$ such that $E_v = \{ \ob^n \in A\}$, and $\{X_v \geq c_2 \beta \log n\} = \{\ob^n \in B\}$. Let  $$\mathcal{S} =\{(x,O) \in K_{k_n}(v) \times B:x \not\in O\}.$$
We define the map $\varphi\colon \mathcal{S} \rightarrow 2^{B_n(v)}$  by $$\varphi(x,O)= O\cup \{x\}. $$ Then for any $(x,O) \in \mathcal{S}$, $\P(\ob^n = \varphi(x,O)) = \frac{1 -\pe}{\pe}\P(\ob^n = O)$. Thus,
$$\P(\ob^n \in \varphi(\mathcal{S})) \leq \frac{1 -\pe}{\pe} \sum_{(x,O) \in \mathcal{S}}\P(\ob^n = O) \leq \frac{1 -\pe}{\pe}|K_{k_n}(v)|\P(\ob^n \in B).$$
By definition, we have $A \subset\varphi(\mathcal S)$. Hence,
$$\P(E_v) =  \P\{ \ob^n \in A\} \leq  (2 k_n +1)^d \frac{1-\pe}{\pe} \P\{ X_v \geq c_2 \beta \log n \}\,. $$
Therefore,
\begin{align*}
  \P(X_v \geq \beta)& \leq \P(G_v \cap \{X_v \geq \beta\}) + \P(G^c_v)\\
                      & \leq \P(E_v) + \P(G^c_v)\\
                      & \leq (2 k_n +1)^d \frac{1-\pe}{\pe} \P\{ X_v \geq c_2 \beta \log n \} +  n^{-2d-1}\\
                      & \leq c_1  k_n^d\P\{ X_v \geq c_2 \beta \log n \} + n^{-2d-1} \,,
\end{align*}
where $c_1>0$ is a constant depending only on $(d, \pe)$.
\end{proof}

\section{Candidate regions for localization} \label{sec:candidate-regions}

Recalling our discussion on proof strategy in Section~\ref{sec:proof-strategy}, in order to show localization it is important to show that all except poly-logarithmic many small regions will be suboptimal compared to some ``best'' small region. Here, to measure the level of ``goodness'' for small regions, we will use principal eigenvalues, which in turn is closely related to survival probabilities for random walk as we show in Lemma~\ref{def:CRlambda}. Thus, it is natural to introduce the following quantiles which measure goodness from the perspective of survival probabilities (in $k_n$ steps):
\begin{equation}
\label{eq:palpha-def}
\begin{split}
	  p_0 &:= \sup\{\beta \geq 0, \P(X_v \geq \beta) \geq n^{-d} k_n^{2d}\log n \},\\
  p_\alpha &:= p_0/(c_2\log n)^\alpha \quad \mbox{for } \alpha \geq 0 ,
\end{split}
\end{equation}
where $c_2$ is chosen such that \eqref{eq:tail} holds. Denote 
\begin{equation}
\label{eq:U-def}
 	\calU_\alpha := \{v\in \Z^d: X_v \geq p_\alpha \}\,.
 \end{equation} Thus, we have that $\calU_0 = \{v\in \Z^d: X_v \geq p_0 \}.$ Heuristically, the hope is that if $\alpha$ is large enough, all regions outside $\calU_\alpha$ will be suboptimal compared to $\calU_0$. In order to make this intuition rigorous, it turns out more convenient to consider principal eigenvalues for small regions. To this end, we introduce the following definition.
 \begin{defn}
\label{def:CRlambda}
	For any site $v \in \Z^d$, we let $\mathcal C_R(v)$  be the connected component in $B_{R}(v) \backslash \ob$ that contains $v$ for $R = k_n (\log n)^2$, and let $\lambda_v$ be the principal eigenvalue of $P|_{\mathcal C_R(v)}$ where $P|_{\mathcal C_R(v)}$ is the transition matrix of simple random walk on $\Z^d$ restricted to $C_R(v)$.
\end{defn}
 In the next lemma, (as announced earlier) we will relate the survival probability $X_v$ to the principal eigenvalue $\lambda_v$, and thus relating the survival probability in $k_n$ steps (i.e., $X_v$) to the survival probability to arbitrary number of steps. 
\begin{lemma}
\label{eigv}
For any $m \geq 1$,
	\begin{equation}
	\label{eq:eigv-1}
		\lambda_v^m\leq \max_x \Pr^x(\xi_{\mathcal C_R(v)} > m) \leq (2R)^{d/2}\lambda_v^{m}\,.
	\end{equation}
In particular,
\begin{equation}
	\label{eq:eigv-2}
	\big(X_v/(2R)^{d/2}\big)^{1/k_n} \leq \lambda_v \leq   \max_{x \in \mathcal C_R(v)} (X_x)^{1/k_n}.
\end{equation}
\end{lemma}
We set $\alo = 3d$ and $\alh = 4d$. By definition there is a clear separation on the level of goodness (in terms of survival probabilities in $k_n$ steps) for typical regions in $\calU_{0}$, $\calU_{\alt}$ and $\calU_{\alh}$ where $\calU_{0}$ contains the ``most desirable'' regions. The level $p_{\alpha_1}$ will be the threshold of candidate regions, while the spacing between $p_0$ and $p_{\alpha_1}$ is used in Lemma~ \ref{firststep} and the spacing between $p_{\alpha_2}$ and $p_{\alpha_1}$ is used in Lemma~\ref{Est-eigv}. By Lemma~\ref{eigv}, such separation can be translated to that in terms of principal eigenvalues (which then controls survival probabilities for arbitrary number of steps).
This motivates the following definition:
\begin{align}\label{eq-def-D}
\begin{split}
\tD_{\lambda} := \{v\in \Z^d: \lambda_v > \lambda  \} \text{ and }\locD& := \{v\in \Z^d: \lambda_v \geq p^{1/k_n}_{\alt}  \}\,.
 \end{split}
 \end{align}
 With preceding definitions, $\locD$ represents candidate regions for localization: indeed, we will show in Section \ref{sec:localization} that random walk will eventually be localized in neighborhoods that are close to $\locD$ (see
\eqref{eq:def-dn} for a formal definition for the union of islands for localization). The remaining section is devoted to proving a number of structural properties for $\locD$ (via structure properties of $\calU_\cdot$), as listed below.
\begin{itemize}
\item We prove Lemma~\ref{goodsiteprob} by a crucial application of Proposition~\ref{tail}, which in turn guarantees that the number of islands in $\calU_{\alpha}$ are at most poly-logarithmic in $n$ --- this is  important for bounding $|D_n|$. 
\item We show in Lemma~\ref{faraway1} that vertices in $\calU_\alpha$ is either close or far away from each other --- this implies that it is costly for the random walk to travel from one good region to another (this is important in the proof of Lemma~\ref{Est-eigv} later).
\item We show in Lemmas~\ref{richsite} and \ref{chemdis} (whose proof uses results in percolation theory) that there exists vertices in $\calU_0$ which are connected to the origin by open paths with lengths which are linear in their Euclidean distances from the origin --- this implies a lower bound on $\Pr(\tau >n)$ by letting the random walk travel to one vertex in $\calU_0$ quickly and stays around it afterwards (see \eqref{eq:lowerbd}).
\item We use Lemma~\ref{eigv} to deduce structural properties on $\tD_\cdot$ from $\calU_\cdot$ --- these are incorporated in Corollary~\ref{lqbzlands} and Lemma~\ref{Emap}.
\end{itemize}

The proofs of Lemma~\ref{eigv}  and the following four lemmas are postponed to Section~\ref{sec:proofs}.
 
\begin{lemma}
\label{goodsiteprob}
We have 
\begin{align*}
n^{-d} k_n^{2d}\log n& \leq \P(X_v \geq p_0) \leq n^{-d}k_n^{4d}\,,\\
\mbox{ and } \quad \P(X_v \geq p_\alpha) &\leq n^{-d}k_n^{(\alpha+4) d}\,.
\end{align*}
\end{lemma}
\begin{lemma}
\label{faraway1}
For any $\alpha \in \N$, with $\P$-probability tending to one there exist no $ u, v \in \calU_{\alpha} \cap B_{2n}(0)$ such that $2 k_n \leq |u - v| \leq n  k_n^{-2(\alpha+5)}$.
\end{lemma}
\begin{lemma}
\label{richsite}
Conditioned on the origin being in an infinite cluster, with $\P$-probability approaching one 
\begin{equation}
	\calU_0 \cap \calC(0) \cap B_{n/k_n}(0) \not = \varnothing\,.
\end{equation}
\end{lemma}

\begin{lemma}
\label{chemdis}
Let $D(u, v)$ be defined as in \eqref{eq-def-chemical}. For $\pe > p_c(\mathbb Z^d)$, there exists a constant $\rho>0$ which only depends on $(d,\pe)$ such that the following holds with $\P$-probability tending to one. For all $u, v \in B_{2n}(0)$
  \begin{equation}
  \label{eq:finite cluster}
  	either~\calC(u) = \calC(0) \ or \ |\calC(u)| \leq (\log n)^3\,,
  \end{equation}
 \begin{equation}
 \label{eq:chemdis}
   D(u,v)\11_{\{u \leftrightarrow v \}} \leq \rho\max(| u - v|, (\log n)^3)\,.
 \end{equation}
\end{lemma}

\begin{cor} We have that
\label{lqbzlands}
\begin{enumerate}[(1)]
	\item With $\P$-probability tending to one, for any $v \in B_{2n}(0)\cap(\calU_{\alh} \cup \D_{p_{\alh}^{1/k_n}})$, $$\big(B_{nk_n^{-14d}}(v) \setminus B_{3R}(v) \big)\cap \big(\calU_{\alh} \cup \D_{p_{\alh}^{1/k_n}}\big) = \varnothing\,.$$
	\item $ k_n^{2d} n^{-d} \leq \P(v \in \locD) \leq k_n^{\alpha + 6} n^{-d}$.
	\item $p_{\alh}^{1/k_n} \geq 1 - {\chi/(\log n)^{2/d}}$ for some constant $\chi$ depending only on $(d,\pe)$.
\end{enumerate}
\end{cor}
\begin{proof}
  It follows from \eqref{eq:eigv-2} that
\begin{equation}
\label{eq:3.9}
\begin{split}
	& \{v \in \D_{p_{\alh}^{1/k_n}}\} \subset \cup_{u \in B_{R}(v)}\{u \in \calU_{\alh}\}\,, \ \\
	& \{v \in \calU_0\} \subset \{v \in \locD\} \subset \cup_{u \in B_{R}(v)}\{u \in \calU_{\alt}\}\,.
\end{split}
\end{equation}
Combining with Lemmas \ref{faraway1} and \ref{goodsiteprob} yields (1) and (2). Combining Lemma \ref{goodsiteprob} and Lemma \ref{lem-lambdac} gives (3).
\end{proof}

The following structural property for $\locD$ will be useful.
\begin{lemma}
\label{Emap}
	With $\P$-probability tending to one, there exists a subset $\locV \subset \locD \cap \calC(0) \cap B_{2n}(0)$ such that
\begin{equation}
\label{eq:def-locv}
\begin{split}
		&\lambda_v = \max \{\lambda_u:{u \in B_{3R}(v)} \} \quad \forall \ v \in \locV\,;\\
	&\locD \cap \calC(0) \cap B_{2n}(0) \subset \cup_{v \in \locV}B_{3R}(v)\,;\\
	&B_{nk_n^{-14d}}(v), v \in \locV \cup\{0\}\text{ are disjoint}\,.\\
\end{split}
\end{equation}\end{lemma}
\begin{proof}[Proof of Lemma \ref{Emap}]
Combining \eqref{eq:3.9} and Lemmas \ref{faraway1}, \ref{goodsiteprob} yields the desired result.
\end{proof}

We will prove in Section \ref{sec:localization} that random walk will eventually be localized in the union of the following islands for some constant $\iota>0$ to be selected:
\begin{equation}
\label{eq:def-dn}
  D_n = \bigcup_{v \in \locV } B_{(\log n)^\iota k_n}(v)\,.
\end{equation}
\begin{proof}[Proof of Theorem \ref{thm-main}: volume of the islands]
Combining Corollary \ref{lqbzlands} (2) and Markov inequality implies that with $\P$-probability tending to one, $|\locD \cap B_{2n}(0)| \leq (\log n)^{100d}$. Then by Lemma \ref{Emap},
	 $$|D_n| \leq (2(\log n)^\iota k_n)^d |\locD \cap B_{2n}(0)| \leq (\log n)^{\iota + 200d}\,,$$
	 and $D_n \cap B_{n (\log n)^{-100d^2}} = \varnothing$.
\end{proof}
\subsection{Proof of Lemmas \ref{eigv}, \ref{goodsiteprob}, \ref{faraway1}, \ref{richsite} and \ref{chemdis}}\label{sec:proofs}
\begin{proof}[Proof of Lemma \ref{eigv}]
 Recall that $P|_{\mathcal C_R(v)}$ is the transition matrix restricted to $\mathcal C_R(v)$. Let $\11_x = (0,\dots,0,1,0,\dots,0) \in \R^{C_R(v)}$ be the vector which takes value $1$ only in the coordinate corresponding to the site $x$, and let $\mathbf{1} = (1,1,\dots,1) \in  \R^{C_R(v)}$. We have
 $$\Pr^x(\xi_{\mathcal C_R(v)} > m) = \11_x^{\tran} (P|_{\mathcal C_R(v)})^m \mathbf{1} \leq \lambda_v^m \sqrt{|\mathcal C_R(v)|} \leq (2 R)^{d/2}\lambda_v^m.$$
 Let $\mu$ be the eigenvector of $P|_{\mathcal C_R(v)}$ corresponding to $\lambda_v$, then $$\sum_{x \in \mathcal C_R(v)}\mu(x)\Pr^{x}(\xi_{\mathcal C_R(v)} > m) =\mu^\tran (P|_{\mathcal C_R(v)})^m \mathbf{1}= \lambda_v^m \sum_{x \in \mathcal C_R(v)}\mu(x)\,.$$
  Hence there exists $x \in \mathcal C_R(v)$ such that $\Pr^{x}(\xi_{\mathcal C_R(v)} > m) \geq \lambda_v^m.$
\end{proof}
\begin{proof}[Proof of Lemma \ref{goodsiteprob}]
By Lemma~\ref{lem-lambdac}, since $\beta_\chi \leq n \beta_{\chi/2}$ we can choose  $\chi = \chi(d,\pe)$ such that for large $n$ 
\begin{equation}
  \P(X_v \geq n\beta_\chi) \geq n^{-d + 1}\,.
\end{equation}
Thus, by definition of $p_0$ we see that for any fixed $\alpha$ and sufficiently large $n$,  \begin{equation}
 \label{palphalowerbound}
p_0 \geq n \beta_\chi \mbox{ and hence }   p_\alpha \geq \beta_\chi\,.
 \end{equation}
This allows us to apply Proposition \ref{tail} with $\beta = p_0$, yielding that
\[
  \P(X_v \geq p_0) \leq c_1 k_n^d n^{-d} k_n^{2d} \log n + n^{-2d-1} \leq  n^{-d} k_n^{4d}.
\]
The left continuity of $\P(X_v \geq x)$ gives$$ \P(X_v \geq p_0) \geq n^{-d} k_n^{2d}\log n.$$
Therefore, for $\beta \in (\beta_\chi, p_0)$ and sufficiently large $n$, \eqref{eq:tail} implies 
$$ \P(X_v \geq \beta) \leq 2c_1 k_n^d \P(X_v \geq c_2 \beta \log n)\,.
$$
Since $p_0 = (c_2 \log n)^\alpha p_\alpha$, applying the above inequality $\alpha$ times yields 
$$
\P(X_v \geq p_\alpha) \leq n^{-d}k_n^{(\alpha + 4)d }. \qedhere
$$
\end{proof}

\begin{proof}[Proof of Lemma \ref{faraway1}]
  For any $u, v \in B_{2n}(0)$ such that $ |u - v| \geq 2k_n $, the events $\{X_v \geq p_{\alpha}\}$ and $\{X_u \geq p_{\alpha}\}$ are independent (since in $k_n$ steps, the random walk will not exit the ball of radius $k_n$). Hence Lemma~\ref{goodsiteprob} yields
  $$\P(X_v \geq p_{\alpha}, X_u \geq p_{\alpha}) \leq n^{-2d} k_n^{2(\alpha+4)d}.$$
Then we complete the proof by enumerating all possible $(u,v) \in B_{2n}(0) \times B_{2n}(0)$ such that $2 k_n \leq |u-v| \leq n  k_n^{-2(\alpha+5)}$.
\end{proof}



\begin{proof}[Proof of Lemma~\ref{chemdis}]
By \cite[Theorem 3]{CCN87} and \cite{GM90b} (see also \cite[Corollary 3]{KZ90}) there exists $C>0$ which only depends on $\pe$ such that for all $m \geq 1$
$$\P( |\mathcal C(v)| = m) \leq \mathrm{e}^{ - C m^{1/2}}.$$
Then for any $v \in \Z^d$
\begin{equation}
  \P((\log n)^3 \leq |\mathcal C(v)| <\infty) 
    \leq  \sum_{m \geq (\log n)^3} \mathrm{e}^{-Cm^{1/2}}
       = o(n^{-d}).
\end{equation}
This proves \eqref{eq:finite cluster}.

By \cite[Theorem 1]{AP96}, we know that for $u, v$ with $| u - v| \geq (\log n)^3$ (the main arguments in \cite{AP96} were written for bond percolation, but as the authors suggest one can verify that the proof adapts to site percolation with minimal changes)
  $$\P(u \leftrightarrow v, D(u,v)>\rho | u - v|) \leq e^{-C| u - v|} \leq n^{-C (\log n)^2},$$
  Hence the event $E_n$ that $D(u,v)\11_{\{u \leftrightarrow v \}} \leq \rho|u - v|$ for all $u, v \in B_{3n}(0)$ with $| u - v| \geq (\log n)^3$ has probability tending to one.

  On the event $E_n$, we consider any $u, v \in B_{2n}(0)$ such that $u \leftrightarrow v$ with $| u - v| < (\log n)^3$. 
In the case $\mathcal C(u) = \mathcal C(0)$, we see from the connectivity that there exists $w\in \mathcal C(0)$ such that $\min(| w - v|, | w - u|)  \in [(\log n)^3 , (\log n)^3 +2]$. Then by triangle inequality $\max(| w - v|, | w - u|) \leq 4 (\log n)^3$. Hence $$D(u,v) \leq D(u,w) + D(v,w) \leq 5\rho (\log n)^3.$$
In the case that $\mathcal C(u) \neq \mathcal C(0)$, It follows from \eqref{eq:finite cluster} that with $\P$-probability tending to one $$D(u,v) \leq |\mathcal C(u)| \leq  (\log n)^3.$$
The proof is completed by adjusting the value of $\rho$.
\end{proof}

\begin{remark}
The work \cite{AP96} improves earlier results of \cite{GM90, GM90b}, where the main objective of \cite{GM90} is to understand certain parabolic problems for the Anderson model with heavy potential.
\end{remark}

\begin{proof}[Proof of Lemma \ref{richsite}]
We say a site $v \in \Z^d$ is reachable if the connected component in $B_{k_n}(v) \backslash \ob$ that contains $v$ is of size at least $k_n$. Then by \eqref{eq:finite cluster}, conditioned on origin being in an infinite cluster, with $\P$-probability approaching one all reachable sites are in $\mathcal C(0)$. Let $\calU^*_0 = \{v \in \calU_0: v\mbox{ is reachable} \}$, it suffices to prove $$\P(\calU^*_0 \cap B_{n/k_n}(0) = \varnothing) \to 0\,.$$

To verify this, we first observe that for each site $v \in \Z^d$ in an infinite cluster,  it connects to $\partial B_{k_n}(v)$ by an open path. Hence the connected component in $B_{k_n}(v) \backslash \ob$ that contains $v$ has at least $k_n$ vertices. As a result,
\begin{equation*}
	 \P(v \mbox{ is reachable}) \geq \theta(\pe)\,.
\end{equation*}
Now by FKG inequality and Lemma~\ref{goodsiteprob},  $$\P(v \in \calU^*_0) \geq \P(v \in \calU_0) \cdot \P(v \mbox{ is reachable}) \geq \theta(\pe) n^{-d}k_n^{2d} \log n.$$ 
Since events $\{v \in \calU^*_0\}$ for $v \in (2k_n+1)\Z^d$ are independent of each other, we have
\begin{align*}
  \P(\calU^*_0 \cap B_{n/k_n}(0) = \varnothing)&
  \leq  \big(1 - \theta(\pe) n^{-d}k_n^{2d} \log n\big)^{(2d^{-1}\lfloor n/k_n \rfloor(2k_n+1)^{-1} - 1)^d} \\
  & \leq n^{-c}
\end{align*}
for some constant $c = c(d,p)$. This completes the proof of the lemma.
\end{proof}

\section{Endpoint localization}\label{sec:localization}

In this section, we prove that conditioned on survival for a long time the random walk will be localized in an island (which we refer to as target island below) in $D_n$, where the target island is chosen randomly (from all the poly-logarithmic many islands in $D_n$) with respect to the random walk. In addition, the target island will be a neighborhood of $\smt$ (see Definition~\ref{def:vt} below and Figure \ref{fig:def} for an illustration), which is the best island that the random walk ever visits.
\begin{defn}
\label{def:vt}
	On event $\{S_{[0,n]} \cap \locD \not = \varnothing\}$, we let $\smt$ be the unique site in $\locV$ (defined in Lemma \ref{Emap}) such that $$S_{t_*} \in B_{3R}(\smt) \quad \mbox{ for }\ t_* = \min\{0 \leq t \leq n: \lambda_{S_t} = \max_{0 \leq i \leq n} \lambda_{S_i}\}\,.$$ Otherwise, we set $\smt := 0$, as in such case the random walk never visits any candidate regions.

	For $v \in \locV$ and constant $\iota>0$ to be selected, we define the hitting time of a neighborhood of $v$ by
\begin{equation}
\label{eq-def-T}
\ft(v)  = \min \{ 0 \leq t \leq n: |S_t - v| \leq (\log n)^{\iota}\}\,.
\end{equation}
\end{defn}
\begin{figure}
	\includegraphics[width=3in]{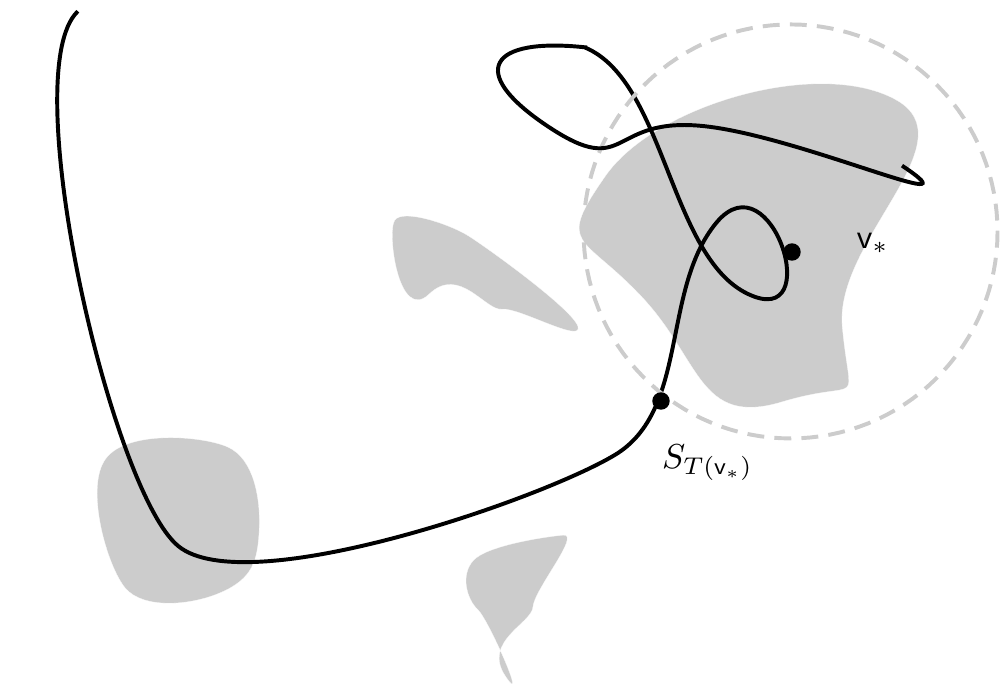}
		\begin{caption}
{The shaded regions are islands in $\loc1D$. The site $\smt$ is the representative of the best island that the random walk ever visits.}
\label{fig:def}
\end{caption}
\end{figure}
The endpoint localization is proved by combining the following two ingredients: 
the random walk will visit $\locD$ with high probability (as shown in Lemma~\ref{firststep});  the random walk will stay in a neighborhood of $\smt$ after getting close to $\smt$ (as shown in Proposition~\ref{endpointloc}).
\begin{lemma}
\label{firststep}
Conditioned on the event that the origin is in an infinite open cluster, 
$$\Pr(\tau_{\loc1D} \leq n \mid \tau > n) \to 1\quad\mbox{ in } \P\mbox{-probability}.\,.$$
 \end{lemma}
\begin{proposition}
\label{endpointloc}
For $\iota$ sufficiently large, conditioned on the event that the origin is in an infinite open cluster, we have that
\begin{equation}
  \Pr(S_{[\ft(\smt),n]} \subset B_{(\log n)^{\iota}k_n}(\smt) \mid \tau > n) \to 1 \quad\mbox{ in } \P\mbox{-probability}.
 \end{equation} 
 \end{proposition} 
\begin{proof}[Proof of Theorem~\ref{thm-main}: endpoint localization] Set $\iota$ to be a sufficiently large constant as in Proposition~\ref{endpointloc}.
Combining Lemma \ref{firststep} and Proposition~\ref{endpointloc} gives
$$\Pr\big(S_n \in D_n\mid \tau > n\big) \to 1\,. \qedhere$$
\end{proof}

In order to prove Lemma~\ref{firststep} and Proposition~\ref{endpointloc}, we first provide upper bound on the probability for the random walk to survive and also avoid $\calU_\alpha$ (or respectively $\tD_\lambda$) as in Lemmas~\ref{lem-avoid-U-alpha} (respectively Lemma \ref{Est-eigv}). Provided with Lemma~\ref{Est-eigv}, Lemma~\ref{firststep} follows from a lower bound  on $\Pr(\tau > n)$, which is substantially larger than (the upper bound on)  $\Pr(\tau_{\ob \cup \loc1D} > n)$. The proof of Proposition~\ref{endpointloc} is yet more complicated, which will employ a careful application of Lemmas \ref{lem-avoid-U-alpha} and \ref{Est-eigv} together with Lemma~\ref{faraway1}.

\subsection{Upper bounds on survival probability}
\begin{lemma}\label{lem-avoid-U-alpha}
For all $\alpha \geq 0$ and $m\geq 1$, we have that for all $v\in \mathbb Z^d$
$$ \Pr^v( \tau_{\calU_{\alpha} \cup\ob} > m) \leq (2R)^{d/2}p_{\alpha}^{m/k_n}\,.$$
\end{lemma}
\begin{proof}
Write $m = jk_n + i$ where $0\leq i<k_n$ and $j\in \mathbb N^*$. By the strong Markov property, we see that for all $v\in \mathbb Z^d$
  \begin{align*}
 \Pr^v( \tau_{\calU_{\alpha} \cup\ob} > m) &= \sum_{x \in (\calU_{\alpha} \cup\ob)^c} \Pr^v(  \tau_{\calU_{\alpha} \cup\ob} > m, S_{m-k_n} = x)\\
    & \leq \sum_{x \in (\calU_{\alpha} \cup\ob)^c} \Pr^v( \tau_{\calU_{\alpha} \cup\ob} > m-k_n, S_{m-k_n} = x) \Pr^x( \tau_{\calU_{\alpha} \cup\ob} >  k_n)\\
    & \leq  \Pr^v( \tau_{\calU_{\alpha} \cup\ob} > m - k_n)\cdot p_{\alpha}.
  \end{align*}
Applying the preceding inequality repeatedly, we get that 
\begin{equation}\label{eq-recursion-avoid-alpha}
\Pr^v( \tau_{\calU_{\alpha} \cup\ob} > m) \leq p_\alpha^j \max_{x\in (\calU_\alpha \cup \ob)^c} \Pr^x(\tau_{\calU_{\alpha} \cup\ob}>i)\,.
\end{equation}
Write $\mathcal C_{\alpha, R}(x) = \mathcal C_R(x) \setminus \calU_\alpha$ and let $\lambda_{\alpha, x}$ be the principal eigenvalue of $P|_{\mathcal C_{\alpha, R}(x)}$. Then, by the same arguments as for Lemma~\ref{eigv}, we deduce that $\lambda_{\alpha, x} \leq (\max_x \Pr^x(\tau_{C_{\alpha, R}(x)}>k_n))^{1/k_n} \leq p_{\alpha}^{1/k_n}$ and then
$$\Pr^x(\tau_{\calU_{\alpha} \cup\ob}>i) \leq (2R)^{d/2}p_\alpha^{i/k_n} \,.$$
Combined with \eqref{eq-recursion-avoid-alpha}, this completes the proof of the lemma.
\end{proof}


\begin{lemma}
\label{Est-eigv}
With $\P$-probability tending to 1 as $n\to \infty$ the following holds. For any $v \in B_n(0)$ and $\lambda > (p_{\alt}/\log n)^{1/k_n}$ and for all $1\leq m \leq n$, we have
\begin{equation}
  \Pr^v(\tau_{\ob \cup \tD_\lambda} > m) \leq R^{3d} \lambda^{m}.
\end{equation}
Here we recall that $\tD_\lambda = \{u \in \Z^d :\lambda_u > \lambda \}$ and $
k_n \geq (\log n)^2.$
\end{lemma}

\begin{proof}
We consider  two scenarios for the random walk. 
In the first scenario, the random walk never enter the region $\calU_{\alh}$. In this case, since for any $u\in \calU_{\alh}^c$ we have $X_u \leq p_{\alh}$, this yields an efficient upper bound. 
In the second scenario, the random walk enters $\calU_{\alh}$ and possibly exits an enlarged neighborhood around $\calU_{\alh}$ and re-enters for multiple times. In this case, we are fighting with the following two factors.
\begin{itemize}
\item The enumeration on the possible times for exiting and re-entering is large (see \eqref{eq:thetal}).
\item  When we estimate the survival probability, we repeatedly use the relation between $X_v$ the principal eigenvalues $\lambda_v$ as in Lemma~\ref{eigv}, and each time we use such a relation we accumulate a certain error factor. As a result, such error factors will grow in the number of times for the random walk to exit an enlarged neighborhood around $\calU_{\alh}$ and then re-enter $\calU_{\alh}$.
\end{itemize}
 In order to beat the preceding two factors, we note that every time the random walk exits an enlarged neighborhood of $\calU_{\alh}$ and re-enters $\calU_{\alh}$, it has to travel for a fair amount of steps outside of $\calU_{\alh}$, due to Lemma~\ref{faraway1}. This leads to a decrement on the survival probability. Such probability decrement, also growing in the number of ``exiting and re-entering'',  is sufficient to beat the enumeration  factor as well as the error factors accumulated when switching between $X_v$ and $\lambda_v$. 

In what follows, we carry out the proof in details following preceding discussions. 
We define stopping times 
$$a_0 = 0\mbox{ and }a_j =  \inf\{ t \geq b_{j-1}: S_t \not \in B_{R}(S_{b_{j-1}}) \mbox{ or }t = m\}\mbox{ for }j \geq 1\,,$$
$$b_j = \inf \{t\geq a_j: S_t  \in \calU_{\alh} \mbox{ or }t = m\}\mbox{ for }j \geq 0\,.$$ 
For all $j\geq 0$ we have $S_t \in \mathcal C_R(S_{b_{j}})$ for $t \in [b_j, a_{j+1} -1]$ and $S_t \not \in \calU_{\alh}$ for $t  \in [a_j, b_j -1]$ (see Figure \ref{fig:Est-eigv}).
\begin{figure}
	\includegraphics[width=3in]{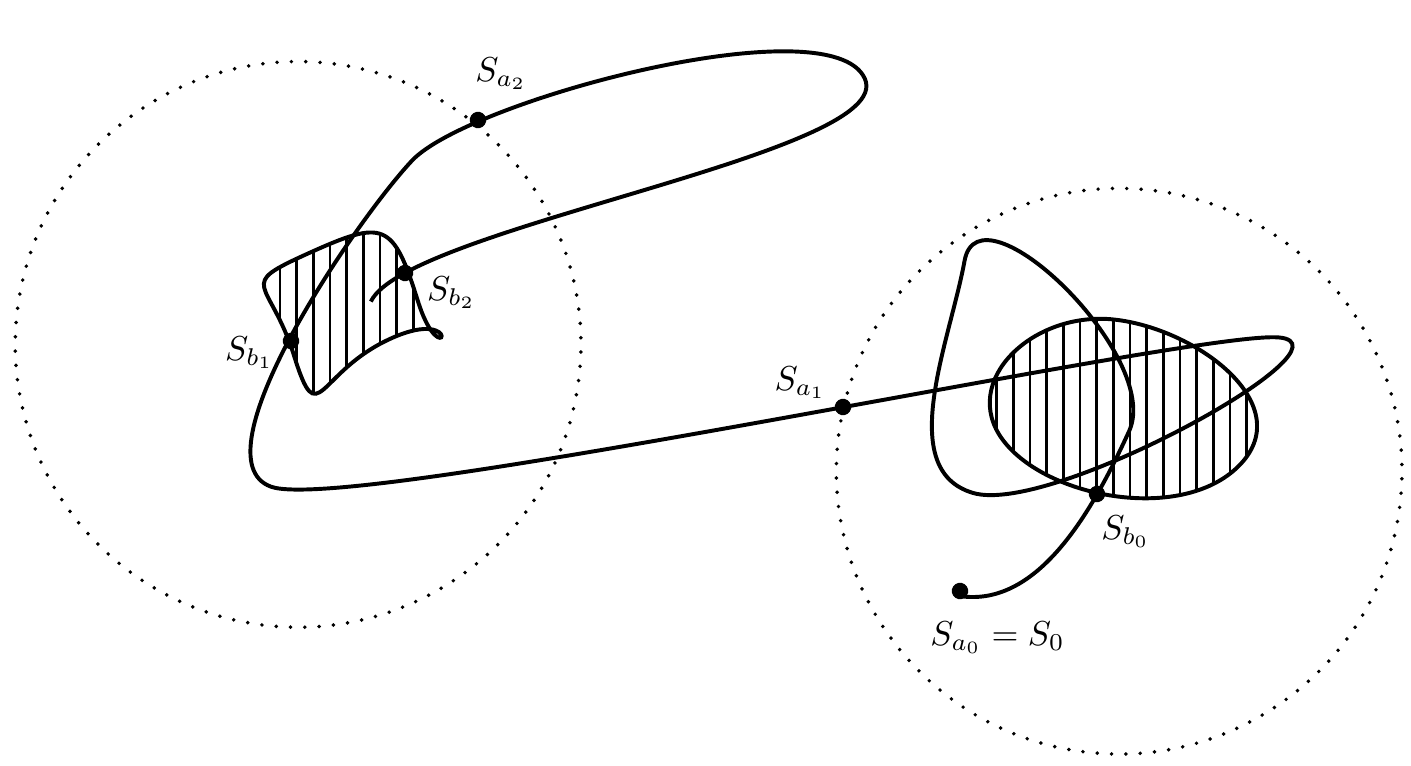}
			\begin{caption}
{The shaded region is $\calU_{\alh}$. Two balls are of radius $R$. The left ball is centered at $S_{b_1}$ and the right one is centered at $S_{b_0}$. Random walk stays in $B_{R}(S_{b_j})$ during  time $[b_j, a_{j+1} -1]$ and stays in $\calU_{\alh}^c$ during time $[a_j, b_j -1]$.}
\label{fig:Est-eigv}
\end{caption}
\end{figure}
 By Lemma \ref{faraway1}, we see that with $\P$-probability approaching 1 we have that $$a_{j} - b_{j-1} \geq R \mbox{ and } b_j - a_{j} \geq R-2k_n \mbox{ for }1 \leq j\leq L-1, $$
where $L = \inf\{j\geq 0:b_j \geq m\}$. We have $L \leq  m/R + 1$. We denote $\Theta_0 = \{(0,m)\}$ and for $1 \leq l \leq  m/R + 1$ define
\begin{align*}
  \Theta_l = \{&(x,y) \in \Z^{(l+1)}\times\Z^{(l+1)} : x_0 = 0,  x_l \leq y_l = m,\\ & x_j<y_{j}<x_{j+1}, y_j - x_j \geq R -2k_n \mbox{ for }j = 0,1,...,l-1 \}.
\end{align*}
Here $\Theta_l$ is the collection of all possible entrance times (to $\calU_{\alh}$) and exit times (from a ball of radius $R$ centered at the entrance point). Then a straightforward combinatorial computation gives that
\begin{align}
\label{eq:thetal}
|\Theta_l| &\leq  \binom{m}{2l} \leq m^{2 l}\,.
\end{align}

For any $m,l \geq 1$ and $(x,y) \in \Theta_l$, we get from \eqref{eq:eigv-1} and Lemma~\ref{lem-avoid-U-alpha} that
\begin{align*}
  &\Pr^v(\tau_{\ob \cup \tD_\lambda} > m, L = l, a_j = x_j, b_j =y_j \mbox{ for }1 \leq j \leq l) \\
  \leq &  \prod_{j=0}^l (2R)^{d} (p_{\alh})^{(y_j - x_j -1 )/k_n}  \lambda^{(x_j - y_{j-1} -1)}\\
  \leq & \lambda^{m} \prod_{j=0}^{l} (p_{\alh})^{-2/k_n}(2R)^{d}\Big(\frac{\log n \cdot  p_{\alh}}{p_{\alt}}\Big)^{(y_j - x_j)/k_n } \,.
\end{align*}
Note that $y_j - x_j \geq R -2k_n \geq k_n \log n$ for $j = 1,2,...,l-1$. Hence for large $n$ (recalling $R = k_n (\log n)^2$ and $p_{\alh} \geq \chi^{k_n/(\log n)^{2/d}}$ as in \eqref{palphalowerbound})
$$\Big(\frac{p_{\alt}}{ \log n \cdot p_{\alh}}\Big)^{(y_j - x_j)/k_n } \geq  n^{10} \mbox{ for }  1\leq j \leq l-1\,.$$
Therefore for $l \geq 2$ and sufficiently large $n$
$$ \Pr^v(\tau_{\ob \cup \tD_\lambda} > m, L = l, a_j = x_j, b_j =y_j \mbox{ for }0 \leq j \leq l) \leq \lambda^m n^{-7(l-1)}.$$
Summing over $l = 2, 3,\ldots, \lfloor m/R \rfloor +1$ and applying \eqref{eq:thetal}, we obtain that for $ m \leq n$
\begin{align}
  \Pr^v(\tau_{\ob \cup \tD_\lambda} > m, L \geq 2) &\leq \lambda^m n^{-4}\,.
   \label{eq-upper-1}
\end{align}
In addition, for $l = 1$ we have
\begin{equation}
  \label{eq-upper-2}
  \Pr^v(\tau_{\ob \cup \tD_\lambda} > m, L = 1) \leq \lambda^m ((p_{\alh})^{-2/k_n}(2R)^{d})^2 \leq  2(2R)^{2d} \lambda^m \,,
\end{equation}
and by Lemma~\ref{lem-avoid-U-alpha} we have
\begin{equation}
  \label{eq-upper-3}
  \Pr^v(\tau_{\ob \cup \tD_\lambda} > m, L = 0) \leq \Pr^v(\tau_{\ob \cup \calU_{\alh}} > m)\leq (2R)^{d/2} \lambda^m\,.
\end{equation}  
Combining \eqref{eq-upper-1}, \eqref{eq-upper-2} and \eqref{eq-upper-3} we completes the proof of the lemma.
\end{proof}

\subsection{Proof of Lemma \ref{firststep} and Proposition \ref{endpointloc}}
\begin{lemma}
\label{loop}
  The following holds with $\P$-probability tending to 1. For all $u, v, w\in B_{2n}(0)$ such that $u \leftrightarrow v, v \leftrightarrow w$ and for any positive number $t$ such that $t - |u-w|_1$ is even, we have
  \begin{equation}
    \Pr^u(S_t = w, \tau>t) \geq (2d)^{-\rho (|u-v|+|v-w|+ R)} \lambda_v^t\,.
  \end{equation}
\end{lemma}
\begin{proof}
This is an immediate consequence of Lemma \ref{chemdis} and \eqref{eq:eigv-1}.
\end{proof}

 \begin{proof}[Proof of Lemma \ref{firststep}]
We first see that reaching $\calU_0$ quickly and staying there afterwards gives a lower bound on $\Pr(\tau > n)$. By Lemmas~\ref{richsite} and \eqref{eq:chemdis}, there exists a site $v_f \in \calU_0$ such that $D(0,v_f) \leq \rho n/k_n.$ It follows from \eqref{eq:eigv-2} that $$\lambda_{v_f} \geq (X_{v_f}/(2R)^{d/2})^{1/k_n} \geq (p_0/(2R)^{d/2})^{1/k_n}\,.$$Then Lemma \ref{loop} implies 
\begin{equation}
\label{eq:lowerbd}
	\Pr(\tau > n) \geq (2d)^{-\rho (n/k_n+1 + R)} \lambda_{v_f}^{n/k_n} \geq  ((2d)^{-2\rho}p_0/(2R)^{d/2})^{n/k_n}.
\end{equation}
By Lemma~\ref{Est-eigv}, we get that 
$\Pr(\tau_{\loc1D\cup \ob}>n) \leq R^{3d}p_{\alo}^{n/k_n} $.
Altogether, we conclude that
\begin{equation*}
  \Pr(\tau_{\loc1D} > n \mid \tau > n) =\frac{\Pr(\tau_{\loc1D \cup\ob} > n)}{\Pr(\tau > n)}
 \leq \frac{R^{3d}p_{\alo}^{n/k_n} }{ ( (2d)^{-2\rho}p_0 (2R_n)^{-d/2})^{n/k_n}}
 = o(1). \qedhere
\end{equation*}
\end{proof}

\begin{proof}[Proof of Proposition \ref{endpointloc}]
Note that 
\begin{equation}\label{eq-decompose-prob-S-t}
\begin{split}
	&\Pr(S_{[\ft(\smt),n]} \not \subset B_{(\log n)^{\iota}k_n}(\smt) ,\tau > n)  \\
= \sum_{v\in \locV}&\Pr( \smt = v, S_{[T(v),n]} \not\subset B_{(\log n)^\iota k_n}(v), \tau>n)\,.
\end{split}
\end{equation}
Consider $v \in \locV\cap B_{n}(0)$. Since $T(v)$ is a stopping time for any fixed $v$, by strong Markov property, we have that 
\begin{align}\label{eq-decompose-T}
\Pr&(\smt = v, S_{[T(v),n]} \not\subset B_{(\log n)^{\iota}k_n}(v), \tau>n) \nonumber \\
  \leq &\Ex[\11_{\{\tau >n-T(v)\}}\Pr^{S_{T(v)}}(\xi_{B_{(\log n)^{\iota}k_n}(v)} \leq m, \tau_{\ob \cup \tD_{\lambda_v}}>m) |_{m = n-T(v)}]\,.
\end{align}
\begin{figure}
	\includegraphics[width=3in]{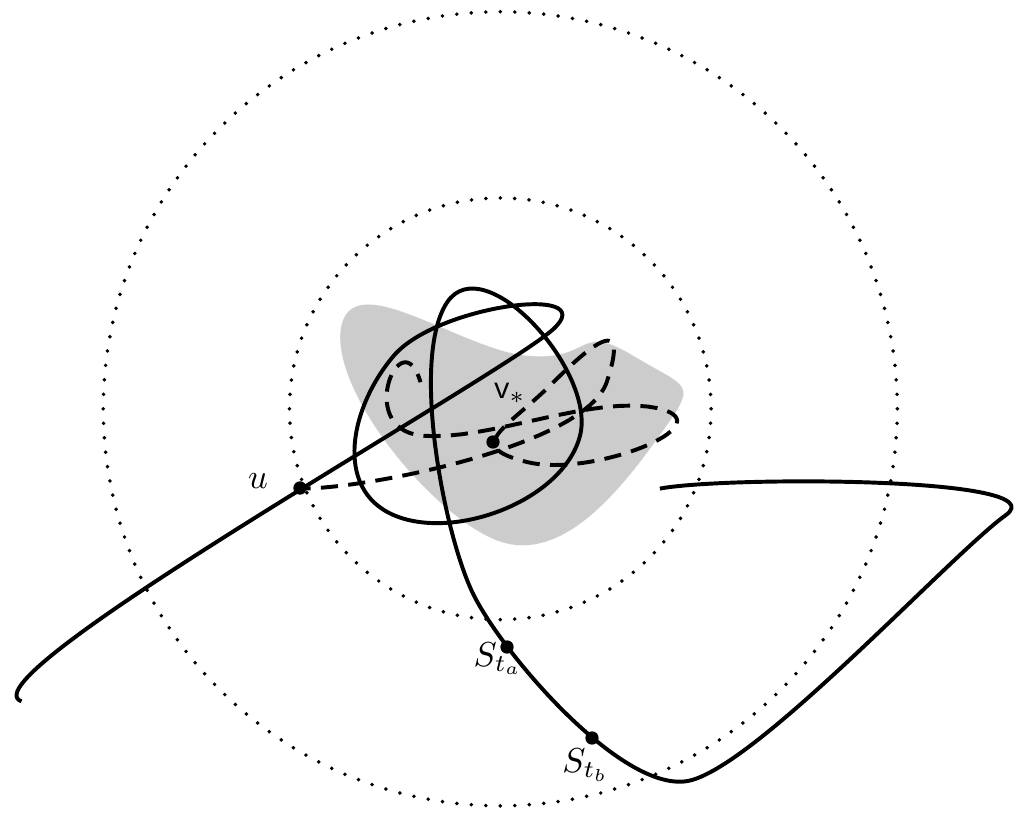}
		\begin{caption}
{The big ball is $B_{(\log n)^{\iota}k_n}(\smt)$ and the small ball is $B_{(\log n)^{\iota}}(\smt)$. If the random walk ever escapes the big ball (solid curve), then it must go through the annulus $B_{(\log n)^{\iota}k_n}(\smt) \setminus B_{3R}(\smt)$ during time $[t_a,t_b]$. But since survival probability in such annulus is very low, the random walk would prefer to stay in the ball (dotted curve).}
\label{fig:endpointloc}
\end{caption}
\end{figure}
We now bound the second term on the right hand side of \eqref{eq-decompose-T}. Suppose that the random walk escapes the ball $B_{(\log n)^{\iota}k_n}(\smt)$ during time $[\ft(v),n]$, then there exists a time interval $[t_a,t_b] \subset [\ft(v),n]$, such that $$t_b-t_a = \lceil (\log n)^{\iota}k_n/2 \rceil \text{ and } S_t \in B_{(\log n)^{\iota}k_n}(\smt) \setminus B_{3R}(\smt) \mbox{ for } t\in [t_a,t_b]$$
(see Figure~\ref{fig:endpointloc}).
 Since Corollary \ref{lqbzlands}~(1) implies that $(B_{(\log n)^{\iota}k_n}(\smt) \setminus B_{3R}(\smt)) \cap \calU_{\alh} = \varnothing$, we get $$S_t \in (\ob \cup \calU_{\alh})^c \mbox{ for } t\in [t_a,t_b]\,.$$ Therefore, for all $v \in \locV$, $u \in  \partial_{i} B_{(\log n)^\iota}(v) \cap \calC(0)$ and $m \in \N$,
\begin{align*}
\Pr^u&(\xi_{B_{(\log n)^{\iota}k_n}(v)} \leq m,\tau_{\ob \cup \tD_{\lambda_v}}>m) \\
& \leq \sum_{t_a = 0}^{m - \lceil (\log n)^{\iota}k_n/2 \rceil}\Pr^u(\tau_{\ob \cup \tD_{\lambda_v}}>m,S_{[t_a,t_b]}\subset (\ob \cup \calU_{\alh})^c)\,.
\end{align*}
Then we get from Lemmas \ref{lem-avoid-U-alpha} and \ref{Est-eigv} that
\begin{align*}
  \Pr^u(\tau_{\ob \cup \tD_{\lambda_v}}>m,S_{[t_a,t_b]}\subset (\ob \cup \calU_{\alh})^c) &\leq (2R)^{10d} \lambda_v^{m-2 -(t_b-t_a)} p_{\alh}^{(t_b-t_a)/k_n}\\
  &\leq (2R)^{10d} \lambda_v^{m-2} (\log n)^{-(\log n)^\iota/2} \,.
\end{align*}
Therefore, we deduce that
\begin{align*}
  \Pr^u&(\xi_{B_{(\log n)^{\iota}k_n}(v)} \leq m,\tau_{\ob \cup \tD_{\lambda_v}}>m) \\
  &\leq  (\log n)^{-(\log n)^\iota/4} (2d)^{-4\rho (\log n)^\iota}\lambda_v^{m}
   \leq (\log n)^{-(\log n)^\iota/4}\Pr^u(\tau>m) \,,
\end{align*}
where we have used the fact which follows from Lemma \ref{loop} that $$\Pr^u(\tau>m) \geq (2d)^{-4\rho (\log n)^\iota}\lambda_v^{m}, \quad \forall \ u \in \partial_{i} B_{(\log n)^\iota}(v) \cap \calC(0)\,.$$ Together with \eqref{eq-decompose-T}, we get
\begin{align*}
  &\Pr(\smt = v, S_{[T(v),n]} \not\subset B_{(\log n)^\iota k_n}(v), \tau>n) 
  \leq  (\log n)^{-(\log n)^\iota/4} \Pr( \tau >n)\,.
\end{align*}
Combined with \eqref{eq-decompose-prob-S-t}, this completes the proof of the proposition by summing over all $v\in \locV\cap B_{n}(0)$	.
\end{proof}

\section{Path localization} \label{pathlocalization}

This section is devoted to the proof of path localization. More precisely, we will show that conditioned on survival the amount of time  the random walk spends before getting close to the island in which it is eventually localized, is at most linear in the Euclidean distance from that island to the origin.

To this end, we consider the loop erasure for the random walk path, i.e., we consider the following unique decomposition for each path $\omega \in \K_{\Z^d}(0,u)$:
\begin{equation}
\label{eq:defLE}
  \omega = l_0 \oplus [\eta_0,\eta_1] \oplus l_1 \oplus[\eta_1,\eta_2]\oplus ...\oplus[\eta_{|\eta|-1},\eta_{|\eta|}] \oplus l_{|\eta|}
\end{equation}
where $l_i \in \K_{\Z^d \setminus\{\eta_0,...,\eta_{i-1} \}}(\eta_i,\eta_i)$ (recall \eqref{eq-def-K}) are the loops erased in chronological order and $\eta = [\eta_0,..,\eta_{|\eta|}]$ is the loop-erasure of $\omega$ denoted by $\eta = \eta(\omega)$ (see \cite[Chapter 9.5]{Lawler10} for more details on loop erasure). 
We first show in Lemma \ref{LoopErasureEst} that in a typical environment the survival probability for the random walk decays exponentially in the length of its loop erasure, which then implies that the loop erasure of the random walk path upon reaching the target island has at most a linear number of steps. 

In light of the preceding discussion, it remains to control the lengths of the erased loops which we consider in the following two cases. 
\begin{itemize}
\item For loops of lengths at most $k_n^{50d}$: we will first show that for a typical environment for majority of the vertices on any self-avoiding path, the survival probability for the random walk started at those vertices up to time $t\leq k_n^{50d}$ decays quickly in $t$ (Lemma \ref{mcqlwzj}); as a consequence we then show in Lemma \ref{nobigloop-1} that it is too costly for the small loops to have a total length super-linear in the length of the loop erasure $|\eta|$. 

\item For loops of lengths at least $k_n^{50d}$: we will first show in Lemma~\ref{eigvest} that except near the target island the random walk does not encounter any other vertex around which the principal eigenvalue is close to that of the target island; as a result we then show in Lemma~\ref{nobigloop-2} that it is too costly to have any big loop.
\end{itemize}

In the rest of the section, we carry out the details as outlined above. 
\begin{defn}
\label{def-Mt}
 Let $\mathcal{M}(t)$ be the collection of sites $v$ such that
\begin{equation}
	\Pr^v(\tau > t) \geq e^{-t/(\log t)^2}\,.
\end{equation}
In addition, we define
 \begin{equation}\label{eq-def-A-t}
A_t(\omega) = \{0\leq i \leq |\eta|: |l_i| = t,\eta_i \not \in \calM(t)\}\,.
\end{equation}
\end{defn}

\begin{lemma}\label{Mtprob}
  There exist positive constants $c_1,c_2$ depending only on $(d,\pe)$ such that for all $t \in \N$
  \begin{equation}
    \P(v \in \calM(t)) \leq c_1 e^{-c_2(\log t)^d}\,.
  \end{equation}
\end{lemma}
\begin{proof}
By Lemmas \ref{cgoodperco} and  \ref{cgoodprob}, for $n \geq 2$
  \begin{equation*}
    \P(\Pr^v(\tau >\lfloor (\log n)^{2/d}\rfloor) \geq c) \leq n^{-\delta(c)}\,.
  \end{equation*}
By a change of variable, there exist constants $c_1,c_2$ depending only on $(d,\pe)$ such that for all $t \in \N$ 
  \begin{equation*}
    \P(\Pr^v(\tau >\lfloor (\log t)^2\rfloor) \geq 1/10) \leq c_1 e^{-c_2(\log t)^d}\,.
  \end{equation*}
Therefore, by a simple union bound we get that
\begin{align*}
\P&(\exists u \in K_{t}(v) \ s.t. \ \Pr^u(\tau >\lfloor (\log t)^2\rfloor) \geq 1/10) 
   \leq  c_3 \exp(-c_4(\log t)^d)\,,
\end{align*}
where $c_3,c_4$ are constants only depends on $(d,\pe)$.

On the event  $\{\Pr^u(\tau >\lfloor (\log t)^2\rfloor) \leq 1/10$ for all $ u \in K_t(v)\}$ (recall the definition in \eqref{eq:def-kbox}), for every $\lfloor (\log t)^2\rfloor$ steps the random walk has at most $1/10$ probability to survive. Thus, 
$$\Pr^{v}(\tau >t) \leq 10^{-t/(2(\log t)^2)}\,.$$ This completes the proof of the lemma.
\end{proof}

\begin{lemma}
\label{mcqlwzj}
  There exists a costant $t_1^* = t_1^*(d,\pe)$ such that the following holds with probability tending to one. 
  For all self-avoiding path $\swp$ started at origin with length $|\swp| \geq n (\log n)^{-100d^2}$and  $t \geq t_1^*$, 
  \begin{equation}
    |\swp\cap\calM(t)| \leq e^{- (\log t)^{3/2}} |\swp| \,.
  \end{equation}
\end{lemma}
\begin{proof}
If $m>n (\log n)^{-100d^2}$ and $\log t \geq (\log m)^{5/9}$, Lemma \ref{Mtprob}  yieds
$$\P(\calM(t) \cap B_m(0) \not= \varnothing) \leq c_1 e^{- c_2(\log t)^{d} + d\log(2m)}\leq e^{- 2^{-1}c_2(\log t)^{d}} \,.$$
Hence, it suffices to prove that for large $t$ and $m$ such that $\log t \leq (\log m)^{5/9}$, we have
  \begin{equation}
  \label{eq:wjnyjdc}
    \P(\max_{\swp \in \calW_{\Z^d,m}(0)}|\swp \cap \calM(t)| \geq  e^{- (\log t)^{3/2}}m) \leq \exp\left(-\exp(- (\log t)^{7/4})m\right)\,,
  \end{equation}
  where $\calW_{\Z^d,m}(0)$ is the collection of self-avoiding path in $\Z^d$ of length $m$ (Note that $\sum_{t: \log t \leq (\log m)^{5/9}}\exp\left(-\exp(- (\log t)^{7/4})m\right)\leq  \exp\left(-\sqrt{m}\right)$ for large $m$).

To this end, we denote $\calV_i =i + (2t+1)\Z^d$ for $i \in \{1,2,...,(2t+1)\}^d$, where $\calV_i$ inherits the graph structure from the natural bijection which maps $v\in \Z^d$ to $i + (2t+1) v \in \calV_i$. Then events $\{x \in \calM(t)\}$ for $x \in \calV_i$ are independent. 
  For any self-avoiding path $\swp$, we know that $\{ x \in \calV_i: \swp \cap K_{t}(x) \not = \varnothing \}$ is a lattice animal (i.e., a connected subset) in $\calV_i$ of size at most $ 3^d |\swp|/t$. Combined with Lemma \ref{Mtprob} and a result on greedy lattice animals proved in \cite[Page 281]{Lee97} (see also \cite{Martin02}), this implies
  \begin{align*}
   & \P(\max_{\swp \in \calW_{\Z^d,m}(0)}|\swp \cap \calV_i \cap \calM(t)| \geq  \exp(- (\log t)^{5/3}) m)\\
     \leq &\P(\max_{\swp \in \calW_{\Z^d,m}(0)}| \{ x \in \calV_i \cap \calM(t): \swp \cap K_{t}(x) \not = \varnothing  \}| \geq \exp(- (\log t)^{5/3})m)\\
   \leq& \exp\left(-2^{-1}\exp(- (\log t)^{5/3})m\right)\,.
  \end{align*}
We complete the proof of \eqref{eq:wjnyjdc} by summing over $i \in \{1,2,...,(2t+1)\}^d$.
\end{proof}

\begin{lemma}
\label{LoopErasureEst}
There exist constants $c \in (0,1), c', r_0>0$ depending only on $(d,\pe)$ such that for any $r_1>r_0$, the following holds for with $\P$-probability at least $1- e^{-c'r_1}$. For all $r>r_1$ and $m \in \N$,
    \begin{equation}
    \Pr( |\eta(S_{[0,m]})|\geq r,\tau > m) \leq c^{r} \,.
  \end{equation}
\end{lemma}
\begin{proof}

By \eqref{eq:wjnyjdc}, we see that there exist constants $C>e^{10}$ and $c',r_0>0$ depending only on $(d,\pe)$ such that for all $r_1 > r_0$, with $\P$-probability at least $1 - \exp\left(-c'r_1\right)$, for all self-avoiding path $\swp$ of length at least $r_1$,
  \begin{equation}
  \label{eq:pathmc}
    |\swp \cap \calM(C)| \leq e^{- (\log C)^{3/2}} |\gamma| \,.
  \end{equation}
We recursively define stopping times $\zeta_0 = 0$,
$$\zeta_i = \inf\{t > \zeta_{i-1}+C: S_t \not \in \calM(C)\}\,. $$
On the event $\{|\eta(S_{[0,m]})|\geq r\}$, since we assumed $r>r_1$, we know from \eqref{eq:pathmc} that $$|S_{[0,m]} \cap \calM(C)^c| \geq |\eta(S_{[0,m]} )\cap \calM(C)^c| \geq |\eta(S_{[0,m]})|(1 - e^{- (\log C)^{3/2}})\,.$$
Let $j = \lfloor r/(2C)\rfloor$. By definition of $\zeta_i$'s and $C>e^{10}$, $$|S_{[0,\zeta_{j}]}\cap \calM(C)^c |\leq (C+1)j + 1 \leq r\tfrac{C+1}{2C}+1 < |\eta(S_{[0,m]})|(1 - e^{- (\log C)^{3/2}})\,.$$ Therefore, we get $\zeta_{j} \leq m\,$. Then by strong Markov property,
\begin{align*}
	  \Pr(|\eta(S_{[0,m]})|\geq r,\tau > m)  &\leq \Pr ( S_{[\zeta_{m},\zeta_{m}+C]}\mbox{ is open }\forall 1 \leq m \leq j-1)\\
  & \leq \left[\exp(-C/(\log C)^2)\ \right]^{r/(2C)-2} \,,
\end{align*}
completing the proof of the lemma.
\end{proof}

\begin{lemma}
\label{nobigloop-1}
Recall definitions in \eqref{eq-def-K} and \eqref{eq-def-D}. There exists a constant $t_2^* = t_2^*(d,\pe)$ such that the following holds with $\P$-probability tending to one. For all $t_2^* \leq t \leq k_n^{50d}$, $u \in \bigcup_{v \in \locV} \left( \partial_{i} B_{(\log n)^\iota}(v) \cap \calC(v)\right)$ and $m \leq n$,
  \begin{equation}
  \Pr(\{\omega \in \K_{\ob^c,m}(0,u) :l_{|\eta|} =\varnothing,|A_t(\omega)| \geq |\eta| t^{-10}\}) \leq  e^{-n^{1/2}} \Pr(\K_{\ob^c,m}(0,u))\,.
  \end{equation}
\end{lemma}
\begin{proof}
For any $\omega$,  we denote \begin{equation}
      \phi(\omega) = \tilde{l}_0 \oplus [\eta_0,\eta_1] \oplus \tilde{l}_1 \oplus[\eta_1,\eta_2]\oplus ...\oplus[\eta_{|\eta|-1},\eta_{|\eta|}] \oplus \tilde{l}_{|\eta|}
  \end{equation}
  where $\tilde{l_i} = l_i$ if $i\notin A_t(\omega)$ and $\tilde l_i = \varnothing$ otherwise. Note that for any $\omega \in \K_{\ob^c,m}(0,u)$ such that $|A_t(\omega)| \geq |\eta| t^{-10}$, $$m -|\phi(\omega)| =t |A_t(\omega)|\geq |\eta|t^{-9}\,.$$ We consider every $\gamma \in \phi( \K_{\ob^c,m}(0,u))$ such that $m -|\gamma| \geq |\eta|t^{-9}$. For large $t$, since $\eta_i \not \in \calM(t)$ for $i \in A_t(\omega)$ and $|\{ i : \tilde l_i = \varnothing\}| \leq |\eta|$, we have
\begin{align*}
    \Pr(\{\omega \in \K_{\ob^c,m}(0,u) : \phi(\omega) = \gamma,l_{|\eta|} =\varnothing\})
    &\leq \Pr(\gamma)\binom{|\eta|}{\frac{m-|\gamma|}{t}}e^{-t(\log t)^{-2}\tfrac{m-|\gamma|}{t}}\\
    & \leq \Pr(\gamma) e^{-(m-|\gamma|)(\log t)^{-3}} \,.
  \end{align*}
In the last inequality, we used the fact that $$(\tfrac{m-|\gamma|}{t})! \geq (\tfrac{m-|\gamma|}{et})^{\tfrac{m-|\gamma|}{t}} \geq (|\eta|t^{-10}e^{-1})^{\tfrac{m-|\gamma|}{t}}\,.$$ In addition, it follows from Lemma \ref{loop} and Corollary \ref{lqbzlands}(3) that
\begin{equation}
\label{eq:loopu}
  \Pr(\K_{\ob^c, m - |\gamma|}(u,u)) \geq (2d)^{-7\rho (\log n)^\iota} e^{-\chi(m-|\gamma|)(\log n)^{-2/d}}\,.
\end{equation}
Note that $(\log t)^3 \leq (\log k_n^{50d}) \lesssim (\log\log n)^3 = o((\log n)^{2/d})$ and by Lemma \ref{Emap} we have $m - |\gamma| \geq |\eta|t^{-9} \geq n (\log n)^{-2000d^2} $. Therefore
\begin{align*}
  \Pr&(\{\omega \in \K_{\ob^c,m}(0,u) : \phi(\omega) = \gamma,l_{|\eta|} =\varnothing\})\\
          \leq& \Pr(\gamma\oplus \K_{\ob^c, m - |\gamma|}(u,u))e^{-2^{-1}|\eta|t^{-10}}\,.
  \end{align*}
We complete the proof of the lemma by summing over all such $\gamma$'s (where the pre-factor of $e^{-\sqrt{n}}$ is a crude bound with room to spare).
\end{proof}

\begin{lemma}
\label{eigvest}
Recall the definition of $\smt$ and $\ft(\smt)$ as in \eqref{eq-def-T}. 
For constant $q>0$, let $U(t) = \cup_{i = 0}^t B_{(\log n)^q}(S_i) \cap \mathcal C(0)$.
Conditioned on the event that the origin is in an infinite open cluster, we have that 
 \begin{equation}\label{eq-eigvest}
	  \Pr( U(n)\setminus B_{3R}(\smt) \subset \tD^c_{(1 - k_n^{-20d}) \lambda_{\smt}} \mid\tau > n) \to 1 \quad\mbox{ in } \P\mbox{-probability}.
 \end{equation} 
 \end{lemma} 
 \begin{remark}
 For purpose of the present article, it suffices to take $U(t) = \cup_{i = 0}^t S_i$; we strengthened the lemma as it may be useful for future application. 
 \end{remark}
\begin{proof}[Proof of Lemma~\ref{eigvest}]
We start with a brief description on the intuition behind \eqref{eq-eigvest}. 
If the random walk hits some local region with the principal eigenvalue close to that near $\smt$ (which is the presumed target island) before time $\ft(\smt)$, then the random walk tends to stay around this local region as opposed to travel all the way to the presumed target island --- since by Lemma~\ref{faraway1} the regions with large principal eigenvalues are far away from each other and thus it is costly for the random walk to travel from one to the other.

Let $a = \inf\{t\geq 0:  \max_{u \in U(t) \setminus B_{3R}(\smt)} \lambda_u   > (1- k_n^{-20d}) \lambda_{\smt}\}$.  Then there exists $x \in \big(B_{2(\log n)^q}(S_a)\cap \calC(0)\big) \setminus B_{3R}(\smt)$ such that and $\lambda_x \geq (1 - k_n^{-20d}) \lambda_\smt$. 
\begin{figure}
	\includegraphics[width=3in]{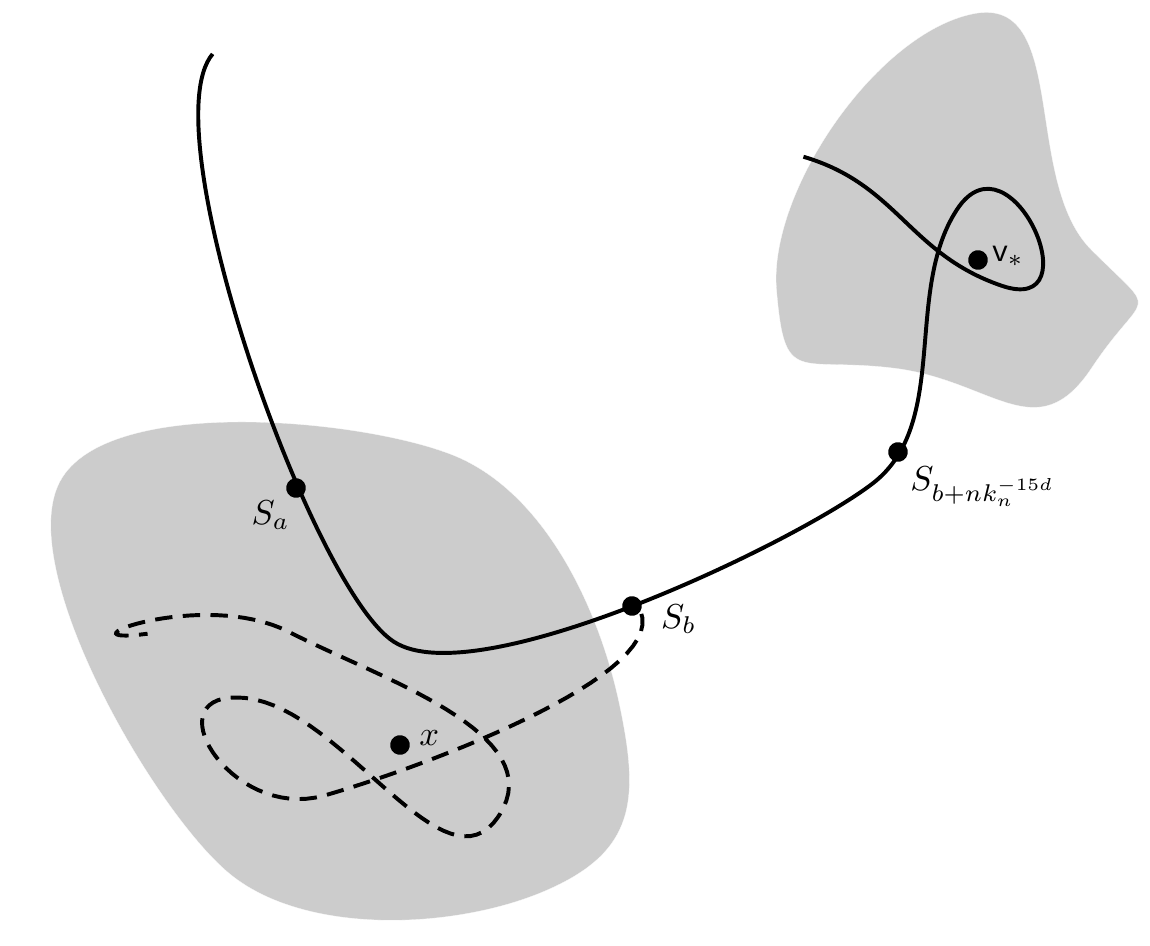}
		\begin{caption}
{The random walk would prefer staying in the neighborhood of $x$ (dotted curve) to going ahead to $\smt$ (solid curve), since survival probability during $[b,b+nk_n^{-15d}]$ is very low. }
\label{fig:eigvest}
\end{caption}
\end{figure}
We restrict to the event $\{\tau_{\locD}<n,S_{[\ft(\smt),n]} \subset B_{(\log n)^{\iota}k_n}(\smt), a<n \}$. Hence, we have$$\lambda_x \geq (1- k_n^{-20d}) \lambda_{\smt}\geq (1- k_n^{-20d}) p_{\alt}^{1/k_n}\geq p_{\alh}^{1/k_n}\,.$$ 
Since $|x - \smt| \geq 3R$, Corollary \ref{lqbzlands}~(1) yields $$|x - \smt| \geq nk_n^{-14d} \ and \ (B_{nk_n^{-14d}}(x) \setminus B_{3R}(x) )\cap \calU_{\alh} = \varnothing \,.$$
 Let $b = \sup\{t\leq n: S_t \in B_{2(\log n)^q}(x) \}$. Then $b+  nk_n^{-15d} \leq n$ and $$S_{[b,b+ nk_n^{-15d}]} \not \in \calU_{\alh}, \  and \  S_{[b+  nk_n^{-15d},n]} \not\in \tD_{\lambda_{\smt}}\,.$$
For any $v \in \locV$ and $x$ such that $\lambda_x \geq (1 - k_n^{-20d}) \lambda_v$, we deduce from the Markov property  and Lemmas~ \ref{lem-avoid-U-alpha}, \ref{Est-eigv} that
   \begin{align}
     \label{eq-seceig-upper}
    & \Pr(b=m,\smt = v,S_a \in B_{2(\log n)^q}(x),\tau>n, a \leq n, S_{[\ft(\smt),n]} \subset B_{(\log n)^{\iota}k_n}(\smt) ) \nonumber \\
         \leq&\Pr(S_m \in \partial_{i} B_{2(\log n)^q}(x),\tau>m)(2R)^{4d} p_{\alh}^{\lfloor nk_n^{-15d} \rfloor/k_n}\lambda_{v}^{n - m- \lfloor nk_n^{-15d} \rfloor}\nonumber\\
        \leq &\Pr(S_m \in \partial_{i} B_{2(\log n)^q}(x),\tau>m)(2R)^{4d}(\log n)^{- \lfloor nk_n^{-15d} \rfloor}\lambda_{v}^{n - m}\,.
   \end{align}

Next, we give a lower bound on survival probability. 
By Lemma \ref{loop}, $\Pr^u(\tau > n-m) \geq (2d)^{-10\rho(\log n)^{q}}\lambda_{x}^{n-m}$. Hence
\begin{align*}
  \Pr(\tau>n)& \geq \Pr(S_m \in \partial_{i} B_{2(\log n)^q}(x),\tau>n)\\
              & \geq \Pr(S_m \in \partial_{i} B_{2(\log n)^q}(x),\tau>m) (2d)^{-10\rho(\log n)^{q}}\lambda_{x}^{n-m}\,.
\end{align*}
Combined with \eqref{eq-seceig-upper} and $\lambda_x \geq (1 - k_n^{-20d}) \lambda_v$, since 
$$(\lambda_v /\lambda_x)^{n-m} \leq (\lambda_v /\lambda_x)^{n} \leq \exp(-nk_n^{-20d})\,, $$
 it yields that 
\begin{align*}
  \Pr(&b=m,\smt = v,S_a \in B_{2(\log n)^q}(x), 
  a \leq n,S_{[\ft(\smt),n]} \subset B_{(\log n)^{\iota}k_n}(\smt), \tau > n) \\& \leq e^{- nk_n^{-16d}}\Pr(\tau>n)\,.
\end{align*}
Summing over $0\leq m \leq n$, $v \in \locV$ and $x \in B_{n}(0)$ such that $\lambda_x \geq (1 - k_n^{-20d}) \lambda_v$, we complete the proof  by Lemma \ref{firststep} and Proposition \ref{endpointloc}.
\end{proof}

\begin{lemma}
\label{nobigloop-2}
For $u \in  \partial_{i} B_{(\log n)^\iota}(v) \cap \calC(v)$ and $\lambda = (1-k_n^{-20d})\lambda_v$ for some $v \in \locV$,  we have that for all $m \leq n$
  \begin{equation}
  \Pr(\{\omega \in \K_{(\tD_\lambda \cup\ob)^c,m}(0,u) :l_{|\eta|} =\varnothing, \max_i|l_i| \geq k_n^{50d}\}) \leq e^{-k_n^{20d}}\Pr(\K_{\ob^c,m}(0,u)) \,.
  \end{equation}
\end{lemma}
\begin{proof}
  For any $\omega$,  we denote \begin{equation}
      \phi(\omega) = \tilde{l}_0 \oplus [\eta_0,\eta_1] \oplus \tilde{l}_1 \oplus[\eta_1,\eta_2]\oplus ...\oplus[\eta_{|\eta|-1},\eta_{|\eta|}] \oplus \tilde{l}_{|\eta|}
  \end{equation}
  where $\tilde{l_i} = l_i$ if
    $|l_i| \leq k_n^{50d}$ and $\tilde l_i = \varnothing$ otherwise. Then for any $\gamma \in \phi( \K_{(\tD_\lambda \cup\ob)^c,m}(0,u))$ with $|\gamma| \neq m$, 
    we deduce from Lemma \ref{Est-eigv} that
\begin{align*}
    &\Pr(\{\omega \in \K_{(\tD_\lambda \cup \ob)^c,m}(0,u) : \phi(\omega) = \gamma,l_{|\eta|} =\varnothing, |\{i: |l_i| > k_n^{50d}\}| = j\})\\
   \leq& \Pr(\gamma )\binom{|\{ i : \tilde l_i = \varnothing\}|}{j} (m - |\gamma|))^jR^{3jd} \lambda^{m-|\gamma|}\,.
      \end{align*}
Summing over $j \leq (m-|\gamma|)/k_n^{50d}$, we get
\begin{align*}
    &\Pr(\{\omega \in \K_{(\tD_\lambda \cup\ob)^c,m}(0,u) : \phi(\omega) = \gamma,l_{|\eta|} =\varnothing\})\\
   \leq& \sum_{j=1}^{\lfloor (m-|\gamma|)/k_n^{50d} \rfloor}\Pr(\gamma )(|\eta|(m - |\gamma|))^j R^{3jd} \lambda^{m-|\gamma|}\\
   \leq & \Pr(\gamma ) n^{3(m-|\gamma|)/k_n^{50d}+2} e^{-(m-|\gamma|)/k_n^{20d}}\lambda_v^{m-|\gamma|}\,.
      \end{align*}
      Note that $\lambda \leq (1-k_n^{-20d})\lambda_v$ and that by Lemma \ref{loop}
$$\Pr(\K_{\ob^c, m - |\gamma|}(u,u)) \geq R^{-3\rho (\log n)^\iota} \lambda_v^{m-|\gamma|}\,.$$
We then get that
\begin{align*}
\Pr&(\{\omega \in \K_{(\tD_\lambda \cup\ob)^c,m}(0,u) : \phi(\omega) = \gamma,l_{|\eta|} =\varnothing\})\\
\leq& \Pr(\gamma\oplus \K_{\ob^c, m - |\gamma|}(u,u))e^{-k_n^{20d}}\,.
\end{align*}
We complete the proof of the lemma by summing over all such $\gamma$'s.
\end{proof}

\begin{cor}
\label{LinearInLoopErasure}
There exists a constant $c = c(d,\pe)$ such that the following holds with $\P$-probability tending to one. If $u \in  \partial_{i} B_{(\log n)^\iota}(v) \cap \calC(v)$ and $\lambda \leq (1-k_n^{-20d})\lambda_v$ for some $v \in \locV$, then for all $m\in \N$, 
  \begin{equation}
    \Pr(\{\omega \in \K_{(\tD_\lambda \cup\ob)^c,m}(0,u) :l_{|\eta|} =\varnothing, m >c |\eta|\}) \leq e^{-k_n^{10d}}\Pr(\K_{\ob^c,m}(0,u))\,.
  \end{equation}
\end{cor}
\begin{proof}
By Lemma \ref{Emap}, with $\P$-probability tending to one, we have $|u| \geq n (\log n)^{-100d^2}$. Then by Lemma \ref{mcqlwzj}, there exists $t^*_1 = t^*_1(d,\pe)$ such that
for all self-avoiding path $\swp$ from $0$ to $u$ and  $t^*_1 \leq t \leq k_n^{50d}$, we have
 \begin{equation*}
    |\swp\cap\calM(t)| \leq t^{-10d} |\swp| \,.
  \end{equation*}
Now, we consider any $\omega \in \K_{(\tD_\lambda \cup\ob)^c,m}(0,u)$ such that $l_{|\eta|} =\varnothing$. If 
$$|A_t(\omega)| \leq |\eta| t^{-10} \mbox{ for }t_2^*\leq t \leq k_n^{50d} \mbox{ and }\max_{0\leq i \leq |\eta|}|l_i| \leq k_n^{50d},$$
for some $t_2^* =t_2^*(d,\pe)$, then for $t^* = \max(t_1^*,t_2^*)$
\begin{align*}
  m &= |\eta| + \sum_{0 \leq i \leq |\eta|}|l_i|\\
    & \leq |\eta| + \sum_{t = 1}^{t^*-1} t |\eta|+\sum_{t = t^*}^{k_n^{50d}}t\left( |A_t(\omega)| +  |\eta \cap \calM(t)| \right)\\
    & \leq \left( 1 + \sum_{t = 1}^{t^*-1} t  + \sum_{t = t^*}^{\infty} t^{-9} +\sum_{t = t^*}^{\infty} te^{- (\log t)^{3/2}}\right) |\eta|\,.
\end{align*}
Combining Lemmas \ref{nobigloop-1} and \ref{nobigloop-2}, we complete the proof of the corollary.
\end{proof}

\begin{proof}[Proof of Theorem \ref{thm-main}: path localization]
We will prove that
   \begin{equation*}
\Pr\Big(T(\smt) \leq c\min(|S_T(\smt)|,n(\log n)^{-2/d}), S_{[T(\smt),n]} \subset D_n\mid \tau > n\Big) \to 1\,.\\
   \end{equation*} 
To this end, applying Lemma \ref{LoopErasureEst} with $r_1 = c_0n (\log n)^{-2/d}$ and combining with \eqref{eq:lowerbd}, we get that there exists $ c_0 = c_0(d,\pe)$ such that
\begin{equation}
  \Pr(|\eta(S_{[0,n]})| > c_0n (\log n)^{-2/d} \mid \tau> n) \to 0\,,
\end{equation}
which implies $\Pr(|S_{\ft(\smt)}| \leq c_0n (\log n)^{-2/d} \mid \tau> n ) \to 1$. (One could also use \eqref{eq:000} in stead of \eqref{eq:lowerbd}.) In light of Proposition \ref{endpointloc}, it remains to prove that there exists $ c = c(d,\pe)$ such that
\begin{equation}
   \Pr({\ft(\smt)} \leq c |S_{\ft(\smt)}| \mid \tau >n) \to 1.
\end{equation}
By Lemma \ref{eigvest}, it suffices to show (we write $\lambda'_v = (1-k_n^{-20d})\lambda_{v}$ below)
\begin{equation} 
\label{eq:pathloc}
  \sum_{v}\sum_{u}\sum_{ m}\Pr(\ft(v) = m, S_{m} = u,\smt = v,\tau_{\tD_{\lambda'_{v}}}>m \mid \tau >n) \to 0\,,
\end{equation}
where the summation is over $v \in \locV$, $u  \in  \partial_{i} B_{(\log n)^\iota}(v)  \cap \calC(v)$ and $c|u| \leq m \leq n$. To this end, we first notice that 
\begin{align}
\label{eq:pathloc-markov}
  &\Pr(\ft(v) = m, S_{m} = u,\smt = v,\tau_{\tD_{\lambda'_{v}}}>m, \tau >n)\nonumber \\ 
    \leq &\Pr(S_m = u, u \not \in S_{[0,m-1]},\tau_{\ob \cup \tD_{\lambda'_{v}}} >m) \Pr^u(\tau > n-m)\,.
\end{align}
At the same time, by Corollary \ref{LinearInLoopErasure} and Lemma \ref{LoopErasureEst} (applied with $r_1 =m/c_1'$), there exist positive constants $c'_1,c'_2$ depending only on $(d,\pe)$ such that for any $v \in \locD$ and $u \in  \partial_{i} B_{(\log n)^\iota}(v) \cap \calC(v)$
\begin{align}
  \Pr&(S_m=u,u \not \in S_{[0,m-1]},m > c'_1 |\eta(S_{[0,m]})| ,\tau_{\ob \cup \tD_{\lambda'_{v}}} >m) \nonumber\\
  &\leq e^{-k_n^{10}}\Pr(S_m = u, \tau > m), \label{eq:pathloc-1} \\
\mbox{ and }  \qquad \Pr& (m \leq c'_1 |\eta(S_{[0,m]})|, \tau >m) \leq  {c_2'}^{m}. \label{eq:pathloc-2}
\end{align}
Then by Lemma \ref{loop} and Corollary \ref{lqbzlands}~(3), we have 
$$\Pr(S_m = u, \tau > m)  \geq (2d)^{-2\rho (\log n)^\iota-\rho|u|}e^{-\chi m/(\log n)^{2/d}}\,.$$ Thus, there exists $c = c(d,\pe)$, such that for all $m > c|u|$
$$\Pr(S_m = u, \tau > m) \geq {c_2'}^{m/2}\,.$$  Combined with \eqref{eq:pathloc-1} and \eqref{eq:pathloc-2}, this gives that
\begin{align*}
&\Pr(S_m = u, u \not \in S_{[0,m-1]},\tau_{\ob \cup \tD_{\lambda'_{v}}} >m) \\
\leq& \Pr(S_m=u,u \not \in S_{[0,m-1]},m > c'_1 |\eta(S_{[0,m]})| ,\tau_{\ob \cup \tD_{\lambda'_v}} >m) \\
   &  + \Pr(m \leq c'_1 |\eta(S_{[0,m]})|, \tau>m) \\
\leq & e^{-2^{-1}k_n^{10}}\Pr(S_m = u, \tau > m)\,.
\end{align*}
Combined with \eqref{eq:pathloc-markov}, this implies
\begin{align*}
  \Pr&(\ft(v) = m, S_m = u,\smt = v,\tau_{\tD_{\lambda'_{v}}}>m, \tau >n) \\
    \leq &e^{-2^{-1}k_n^{10}}\Pr(S_m = u, \tau > m) \Pr^u(\tau > n-m)\\
    = &e^{-2^{-1}k_n^{10}} \Pr(S_m = u, \tau > n) \,.
\end{align*}
Summing over $u,v$ and $m \geq c|u|$, we complete the verification of \eqref{eq:pathloc}.
\end{proof}

\end{document}